\documentclass[12pt,psamsfonts]{amsart}
\usepackage{amsmath}
\usepackage{amsthm}
\usepackage{amssymb}
\usepackage{amscd}
\usepackage{amsfonts}
\usepackage{amsbsy}
\usepackage{graphicx}
\usepackage[dvips]{psfrag}
\usepackage{array}
\usepackage{color}
\usepackage{epsfig}
\usepackage{url}
\usepackage{overpic}
\usepackage{epstopdf}
\usepackage{float}
\usepackage{enumerate}
\usepackage{multirow}
\usepackage{rotating}
\usepackage{hyperref}

\usepackage{mathtools}
\usepackage[left=1.5cm,top=1.5cm,right=1.5cm,bottom=1.5cm]{geometry}
\usepackage[scr]{rsfso}

\graphicspath{{./Figuras/}}

\usepackage{mathpazo}
\usepackage[english]{babel} 
\usepackage[utf8]{inputenc}
\usepackage{fancyhdr}
\usepackage{tkz-euclide}
\usetikzlibrary{intersections}
\usepackage{enumerate}
\usepackage{indentfirst}
\usepackage{geometry}
\usepackage[all]{xy}
\usepackage{calc,color}

\theoremstyle{definition} 
\newtheorem{theo}{Theorem}[section] 
\newtheorem{lem}[theo]{Lemma} 
\newtheorem{cor}[theo]{Corollary} 
\newtheorem{prop}[theo]{Proposition} 
\newtheorem{defi}[theo]{Definition} 
\newtheorem{rem}[theo]{Remark} 
\newtheorem{example}[theo]{Example} 
\numberwithin{equation}{section}
\newcommand{\delim}[3]{\left#1 #3 \right#2}  
\newcommand{\pin}[1]{\delim{\langle}{\rangle}{#1}} 
  
\def \tn {\textnormal}

\begin{document}
\renewcommand{\arraystretch}{1.5}

\title[Convergence of attractors]{The effect of perturbations on the convergence of attractors for reaction-diffusion equations concerning variations of nonlinear boundary conditions}
\author{Flank D. M. Bezerra}
\address{Federal University of Para\'{\i}ba, 58051-900, Jo\~{a}o Pessoa PB, Brazil.} 
\email{flank@mat.ufpb.br}

\author{Marcone C. Pereira}
\address{Instituto de Matem\'atica e Estat\'istica da Universidade de S\~ao Paulo, 05508-090, S\~ao Paulo SP, Brazil.}
\email{marcone@ime.usp.br}

\author{Leonardo Pires}
\address {State University of Ponta Grossa, 84030-900, Ponta Grossa PR, Brazil.}
\email{lpires@uepg.br}



\maketitle

\begin{abstract}
This paper presents estimates of the convergence of asymptotic dynamics of reaction-diffusion equations with nonlinear boundary conditions. We show how the convergence of the global attractors can be affected by the variations of diffusion coefficients, boundary conditions, and vector fields.
\end{abstract}

\tableofcontents

\section{Introduction}

The continuity of attractors for small perturbations of semilinear parabolic equations is a fundamental property to understand the qualitative properties present in physical models described by reaction-diffusion equations. Under small changes in the initial data and parameters, a precise model of differential equations must have the permanence and continuity of bounded solutions defined for all time. In addition, the asymptotic behavior of the solutions at infinite must be understood. These topics have been addressed in several works where conditions have been presented to guarantee the existence and continuity of global attractors for semigroups generated by solutions of parabolic equations,  see e.g. \cite{Babin1992,A.N.Carvalho2010,haleetall} and \cite{Hale1988}. The recent works \cite{Arrieta,Santamaria2014,Santamaria2017,Santamaria2018,Carvalho2010,Lpires1,Carvalho,PPires2024,Pires2021,PR2023} and \cite{Pires2023} have been obtained estimates on the continuity of the dynamics of a one-parameter equation when it converges to a limiting equation. In this context, important convergences, such as the permanence of equilibria, local invariant manifolds, and, global attractors have been estimated by a function involving the parameter that appears in the perturbed equation. Thus, it has shown that the stability of well behavior perturbations can be measured, in such a way, that the continuity of the problem produces rates of convergence, as fast as the parameters of the perturbed equation converge to the parameter of the limiting unperturbed equation. It is important to observe that the above recent papers have just considered one parameter, and except \cite{PPires2024}, they have worked without nonlinear boundary conditions and potentials. 
 
In this paper, we show how these rates of convergence can be affected by the perturbations of diffusion coefficients, boundary conditions, and vector fields. Therefore,  we extend the existing results by showing that if we add more parameters to the equation, the rate of convergence will change in a way that continuity will be preserved. We show explicitly how these new terms added in the equation contribute to the speed of convergence of attractors. In particular, if the problem is one-dimensional, then we significantly improve the rates of convergence obtained in \cite{Arrieta}, using the fact that a scalar reaction-diffusion equation generates a Morse-Smale semigroup in the phase space.

We consider the following family of reaction-diffusion equations with nonlinear boundary conditions and potentials, 
\begin{equation}\label{eq_reaction_diffusion}
\begin{cases}
\partial_tu^\varepsilon-\mbox{div}(p_\varepsilon(x)\nabla u^\varepsilon)+(\lambda_1+V_\varepsilon(x))u^\varepsilon=f^\varepsilon(u^\varepsilon)& \tn{in}\ \Omega\times(0,\infty), \\
\dfrac{\partial u^\varepsilon}{\partial \vec{n}_{p_\varepsilon}}+(\lambda_2+b_\varepsilon(x))u^\varepsilon=g^\varepsilon(u^\varepsilon)& \tn{on}\ \Gamma\times[0,\infty),
\end{cases}
\end{equation}
where $0\leqslant\varepsilon\leqslant\varepsilon_0\leqslant 1$, $\Omega\subset\mathbb{R}^N$ is a bounded Lipschitz domain with $\Gamma = \partial \Omega$, the boundary of $\Omega$, and $\vec{n}$ is the outward unitary normal vector for $\Gamma$. Here, $\frac{\partial u}{\partial \vec{n}_{p_\varepsilon}}=p_\varepsilon(x) \nabla u \cdot \vec{n}$ denotes the conormal derivative and we choose $\lambda_1,\lambda_2 \in\mathbb{R}$ sufficiently large such that 
\begin{equation} \label{condlam}
\tn{essinf}_{x\in\Omega}\{\lambda_1+V_\varepsilon\}\geqslant m_0
\quad \textrm{ and } \quad 
\tn{essinf}_{x\in\Gamma}\{\lambda_2+b_\varepsilon\}\geqslant m_0
\end{equation} 
for some $m_0>0$. The nonlinear terms $f^\varepsilon$ and $g^\varepsilon$ are $\mathcal{C}^2$ functions, uniformly bounded with bounded derivatives. The potentials $V_\varepsilon(x)$ and $b_\varepsilon(x)$ are given functions on $\Omega$ and $\Gamma$, respectively. The diffusion coefficients $p_\varepsilon$ are $C^1$ strictly positive functions on $\Omega$, such that, for $\varepsilon\in [0,\varepsilon_0]$,
\begin{equation}\label{conv_p}
0<m_0\leqslant p_\varepsilon(x)\quad\tn{ and }\quad \|p_\varepsilon-p_0\|_{L^\infty(\Omega)}\to 0\tn{ as } \varepsilon\to 0^+.
 \end{equation}

The potential terms verify 
%\begin{equation}\label{conv_v}
%\|V_\varepsilon-V_0\|_{L^{\infty}(\Omega)}\leqslant \eta(\varepsilon)  \quad\tn{and}\quad \|b_\varepsilon-b_0\|_{L^{\infty}(\Gamma)}\leqslant \tau(\varepsilon), 
%\end{equation}
\begin{equation}\label{conv_v}
\|V_\varepsilon-V_0\|_{L^{p}(\Omega)}\leqslant \eta(\varepsilon)  \quad\tn{and}\quad \|b_\varepsilon-b_0\|_{L^{q}(\Gamma)}\leqslant \tau(\varepsilon), 
\end{equation}
where $p=\frac{1}{2}-\frac{1}{r_0}$ with $r_0\in (2,r_0^\ast]$ where $r_0^\ast$ is the critical exponent for the inclusion $H^1(\Omega)\subset L^{r_0}(\Omega)$ and  $q=\frac{1}{2}-\frac{1}{r_1}$ with $r_1\in (2,r_1^\ast]$ where $r_1^\ast$ is the critical exponent for the inclusion $H^\frac{1}{2}(\Gamma)\subset L^{r_1}(\Gamma)$ and, $\eta,\tau:[0,\varepsilon_0]\to [0,\infty)$ are functions such that $\eta(\varepsilon),\tau(\varepsilon)\to 0$ as $\varepsilon\to 0^+$. 
%It is possible to consider $V_\varepsilon$ and $b_\varepsilon$ in more general $L^p$ spaces as we can see in \cite{Rodriguez-Bernal1998} but in order to simplify the arguments we will consider convergences only in $L^\infty$.
  
We denote $\tilde{f}^\varepsilon:H^1(\Omega)\to L^2(\Omega)$ and $\tilde{g}^\varepsilon:H^{\frac{1}{2}}(\Gamma)\to L^2(\Gamma)$  the Nemitsk\u{\i}i operators associated with the functions $f^\varepsilon$ and $g^\varepsilon$, respectively. It is well-known that $\tilde{f}^\varepsilon$ and $\tilde{g}^\varepsilon$ are continuously differentiable. We assume the following convergences
\begin{equation}\label{conv_f}
\sup_{u\in D_1}\{\|\tilde{f}^\varepsilon(u)-\tilde{f}^0(u)\|_{L^2(\Omega)},\|(\tilde{f}^\varepsilon)^\prime(u)-(\tilde{f}^0)^\prime(u)\|_{\mathcal{L}(H^1(\Omega),L^2(\Omega))}\}\leqslant\kappa(\varepsilon) 
\end{equation}
and
\begin{equation}\label{conv_g}
\sup_{u\in D_2}\{\|\tilde{g}^\varepsilon(u)-\tilde{g}^0(u)\|_{L^2(\Gamma)},\|(\tilde{g}^\varepsilon)^\prime(u)-(\tilde{g}^0)^\prime(u)\|_{\mathcal{L}(H^{\frac{1}{2}}(\Gamma),L^2(\Gamma))}\} \leqslant \xi(\varepsilon),
\end{equation}
where $D_1\subset H^1(\Omega)$, $D_2\subset H^1(\Gamma)$  are bounded set, $\kappa,\xi:[0,\varepsilon_0]\to [0,\infty)$ are functions such that  $\kappa(\varepsilon),\xi(\varepsilon)\to 0^+$ as $\varepsilon\to 0^+$. 

Throughout the text we will denote the function $f^\varepsilon$ and its  Nemitsk\u{\i}i operators associated $\tilde{f}^\varepsilon$ by the same notation $f^\varepsilon$. The same for the function $g^\varepsilon$.

In order to obtain the well-posedness of \eqref{eq_reaction_diffusion} we assume in addition to \eqref{conv_f} and \eqref{conv_g} some growth and dissipative conditions in $f^\varepsilon$ and $g^\varepsilon$ state latter (see \eqref{diss_cond} and \eqref{grwoth_cond}).

We are concerned with the continuity of the dynamics of  \eqref{eq_reaction_diffusion} as $\varepsilon\to 0^+$. Notice that the assumptions \eqref{conv_p}-\eqref{conv_g} makes the equation \eqref{eq_reaction_diffusion} a regular perturbation of a limiting problem associated with the limiting parameter $\varepsilon=0$ in such a way that, the convergence of potentials and fields are measured by the parameter functions $\tau(\varepsilon),\kappa(\varepsilon)$ and $\xi(\varepsilon)$. 

In order to guess the limiting problem as $\varepsilon\to 0$, we informally proceed as follows. The first equation in \eqref{eq_reaction_diffusion} reads     
\[
\partial_tu^\varepsilon-\mbox{div}(p_\varepsilon(x)\nabla u^\varepsilon)+(\lambda_1+V_\varepsilon(x))u^\varepsilon=f^\varepsilon(u^\varepsilon)\ \mbox{in}\ \Omega\times(0,\infty),
\] 
and from the convergence properties \eqref{conv_p}, \eqref{conv_v} and \eqref{conv_f}, we  have the limiting equation, as $\varepsilon\to 0^+$,
\[
\partial_tu^0-\mbox{div}(p_0(x)\nabla u^0)+(\lambda_1+V_0(x))u^0=f^0(u^0)\ \mbox{in}\ \Omega\times(0,\infty).
\]
In the same way, we have 
\[
\frac{\partial u^\varepsilon}{\partial \vec{n}_{p_\varepsilon}}+(\lambda_2+ b_\varepsilon(x))u^\varepsilon=p_\varepsilon(x) \nabla u^\varepsilon \cdot \vec{n} +(\lambda_2+b_\varepsilon(x))u^\varepsilon=g^\varepsilon(u^\varepsilon)\ \mbox{on}\ \Gamma\times[0,\infty),
\]
and from the convergence properties \eqref{conv_p}, \eqref{conv_v} and \eqref{conv_f} we have the limiting boundary condition,  as $\varepsilon\to 0^+$,
\[
\frac{\partial u^0}{\partial \vec{n}_{p_0}}+(\lambda_2+b_0(x))u^0=p_0(x) \nabla u^0 \cdot \vec{n}_0 +(\lambda_2+b_0(x))u^0=g^0(u^0)\ \mbox{on}\ \Gamma\times[0,\infty).
\] 

In the next sections, we will show that the nonlinear semigroup set by \eqref{eq_reaction_diffusion} is right continuous at $\varepsilon=0$.  We will prove for \eqref{eq_reaction_diffusion} the convergence of equilibrium points, local invariant unstable manifolds and global attractors, as $\varepsilon\to 0^+$. The main result of this paper is to prove that all these convergences can be estimated by 
\begin{equation}\label{final_est_1}
\Big[\|p_\varepsilon-p_0\|_{L^\infty(\Omega)}+\eta(\varepsilon)+\tau(\varepsilon)+\kappa(\varepsilon)+\xi(\varepsilon)\Big]^{l},
\end{equation}
as $N\geqslant 2$, for some $0<l<1$. On the other side, if $N=1$, we can show that the convergences can be estimated by
\begin{equation}\label{final_est_2}
\Phi(\epsilon)  \log \Phi(\epsilon) 
\end{equation}
where
\begin{equation}
\Phi(\epsilon) = \|p_\varepsilon-p_0\|_{L^\infty(\Omega)}+\eta(\varepsilon)+\tau(\varepsilon)+\kappa(\varepsilon)+\xi(\varepsilon).
\end{equation}
Notice that \eqref{final_est_2} goes to zero faster than \eqref{final_est_1} as $\varepsilon\to 0^+$ and both are composed by the convergence of the potentials, boundary terms, and vector fields. 

The terms $l$ and $\log$ in estimates above in fact represent a loss in the rates of convergence. A result without these terms is still an open problem which can be called an optimal resolvent rate problem. We already know that in the situation of a family of parabolic problems whose asymptotic behavior is dictated by a system of ordinary differential equations, it is possible to obtain the optimal resolvent rate, see \cite{Lpires1}. In addition, constructing explicit estimates that give rise to the rate of convergence of attractors can be hard-working as we can see in \cite{Carvalho}, where an elliptic problem involving a divergente operator with localized large diffusion has been studied. It is interesting to observe that, for a class of problems that generate gradient semigroups, the global attractor is given by a union of unstable manifolds of hyperbolic equilibrium points. Also, the unstable manifold attracts exponentially. Thus, we hope that the rate of convergence of attractors will be exponential. It can be called an optimal exponential rate problem. Both optimal resolvent and exponential rates, as far as we know, are open problems.

To obtain the rate of convergence \eqref{final_est_1} and \eqref{final_est_2} for equation \eqref{eq_reaction_diffusion} we organize this paper as follows. In  Section \ref{Sec2}, we introduce the appropriate functional setting to deal with problem  \eqref{eq_reaction_diffusion} and we state the results on the well-posedness, regularity, and existence of global attractors.  In Section \ref{Sec3}, we show the convergence with a rate of the equilibrium points. Section \ref{Sec5}, is reserved to obtain the convergence of linear and nonlinear semigroups and to obtain one part of the rate of convergence. In Section \ref{Sec6}, we study the local unstable manifold around an equilibrium point through local linearization. In  Section \ref{Sec7}, we present our main result on the rate of convergence of attractors. In Section \ref{Sec8}, we deal with the special one-dimensional case.  We conclude with two appendices: \ref{SecA} is related to the general theory of the rate of convergence of attractors and  \ref{SecB} deals with the Shadowing Theory and its applications in the convergence of attractors. 

\section{Functional setting}\label{Sec2}

In this section, we summarize some already-known results on the existence of global attractors, their uniform bounds, and convergence under diffusion perturbation. We introduce the functional framework to study \eqref{eq_reaction_diffusion} and, we state a fundamental result on the equivalence between the norms of the fractional powers space and the Sobolev space $H^1(\Omega)$. We end the section obtaining the convergence of the nonlinearities taking values in a negative power Sobolev space.

The natural form to study  \eqref{eq_reaction_diffusion} is as an evolution equation in the Sobolev space $H^s(\Omega)$, $s>0$ whose dual space we denote by $H^{-s}(\Omega)$. 
We recall that $H^s(\Omega)$, $s>0$, can be defined as the fractional power space through the Laplace operator with homogeneous Newmann boundary conditions (see for instance \cite{Henry1980, Yagi}). In fact, we have from \cite[Theorem 1.35 and Corollary 2.4]{Yagi} that  $H^s(\Omega) = D((-\Delta + 1)^{s/2})$, $0 \leq s \leq 1$ where $\Delta$ is the Laplacian operator with Neumann homogeneous boundary condition $\Delta:D(\Delta)\subset L^2(\Omega)\to L^2(\Omega)$ with 
$$
D(\Delta)= \left\{ u\in H^2(\Omega): \frac{\partial u}{\partial\vec{n}}=0\hbox{ on } \partial \Omega \right\}.
$$

The duality between these spaces will be denoted by $\pin{\cdot,\cdot}_{-s,s}$. In particular, the scalar product in $L^2(\Omega)$ will be denoted by $\pin{\cdot,\cdot}$. 
For the boundary terms we consider the trace space $H^s(\Gamma)$ and the trace operator  $\gamma_{tr}:H^s(\Omega)\to H^{s-\frac{1}{2}}(\Gamma)$, for $s>\frac{1}{2}$. 
%Since $H^{\frac{1}{2}}(\Gamma)\hookrightarrow L^2(\Gamma)\hookrightarrow H^{-\frac{1}{2}}(\Gamma)\hookrightarrow H^{-1}(\Gamma)$ we define, for $f\in H^1(\Omega)$
%$$
%\pin{\gamma_{tr}(f),\varphi}_{-1,1}=\int_\Gamma f\varphi\,d\sigma,\quad\forall\,\varphi \in H^1(\Omega).
%$$

%We also consider the normal derivative operator relative to the normal divergent operator $-\mbox{div}(p_\varepsilon(x)\nabla u)$, defined as follows: if
%$
%u\in \{z\in H^1(\Omega):-\mbox{div}(p_\varepsilon(x)\nabla z)\in L^2(\Omega)\}
%$
%then $\frac{\partial u}{\partial n_\varepsilon}\in H^{-\frac{1}{2}}(\Gamma)$ and it is defined as
%\begin{equation}\label{conormal_derivative}
%\pin{\frac{\partial u}{\partial n_\varepsilon},\gamma_{tr}(v)}_{-\frac{1}{2},\frac{1}{2}}=\int_\Gamma \mbox{div}(p_\varepsilon(x)\nabla u)v\,dx=-\int_\Gamma p_\varepsilon(x)\nabla u\nabla v \,dx,
%\end{equation}
%for every $v\in H^1(\Omega)$.

We define the unbounded linear operator $A_\varepsilon:D(A_\varepsilon)\subset L^2(\Omega)\to L^2(\Omega)$ by 
\begin{equation}\label{ell_operator}
A_\varepsilon u=-\mbox{div}(p_\varepsilon(x)\nabla u)+(\lambda+V_\varepsilon)u
\end{equation}
with domain 
\begin{equation}\label{ell_operator2}
D(A_\varepsilon)=\Big\{u\in H^2(\Omega):\frac{\partial u}{\partial n_{p_\varepsilon}}=0\ \mbox{on}\ \Gamma \Big\}.
\end{equation}

The operator $A_\varepsilon$ is selfadjoint and has compact resolvent in $L^2(\Omega)$. Moreover, if $\mu>\lambda^\varepsilon_1$, the first eigenvalue of $A_\varepsilon$, then the operator $\mu I+A_\varepsilon$ is positive. Indeed, if we denote by $X^\beta_\varepsilon$ its fractional power space, we have again from \cite[Theorem 1.35 and Corollary 2.4]{Yagi} that $X_\varepsilon^\beta = H^{2\beta}(\Omega)$, for $0\leqslant \beta\leqslant{1}/{2}$. By duality $H^{-2\beta}(\Omega) = X_\varepsilon^{-\beta}$, also for $0< \beta\leqslant {1}/{2}$. In particular, $X_\varepsilon^1=D(A_\varepsilon)$, $X_\varepsilon^{{1}/{2}}=H^1(\Omega)$, $X_\varepsilon^0=L^2(\Omega)$,  $X_\varepsilon^{-{1}/{2}}=H^{-1}(\Omega)$. 

Notice that our choice of $\lambda_1$ in \eqref{eq_reaction_diffusion} is not restrictive, since we can translate the operator $A_\varepsilon$ for a parameter $\mu$ so that $\mu+A_\varepsilon$ is positive and then the fractional power space is defined. Thus, we fix $\lambda_1$ satisfying \eqref{condlam} just to write the inner product in a clear form to study the variational formulation of the elliptic divergent operator.
In fact, the operator $A_\varepsilon$ is the realization in $L^2(\Omega)$ of the canonical isomorphism between $H^1(\Omega)$ and its dual  $H^{-1}(\Omega)$, defined by the bilinear form $a_\varepsilon :H^1(\Omega)\times H^1(\Omega)\to \mathbb{R}$ given by
\begin{equation}\label{bilinear_form}
a_\varepsilon(u,v) = \int_\Omega p_\varepsilon(x)\nabla u\nabla v\,dx+\int_\Omega (\lambda_1+V_\varepsilon)uv\,dx.  
\end{equation}

We have $\pin{A_\varepsilon u,v}_{-1,1} = a_\varepsilon(u,v)$ with $a_\varepsilon$ being continuous, symmetric and uniformly coercive by the following lemma. 
\begin{lem}\label{equivalence_norm}
The inner product in $X_\varepsilon^\frac{1}{2}$ given by
$$
\begin{gathered}
\pin{u,v}_{X_\varepsilon^\frac{1}{2}}=\int_\Omega p_\varepsilon(x)\nabla u\nabla v\,dx+\int_\Omega (\lambda_1+V_\varepsilon)uv\,dx+\int_\Gamma(\lambda_2+ b_\varepsilon(x)) uv\,d\sigma \\
= \pin{A_\varepsilon u,v}_{-1,1} +\int_\Gamma(\lambda_2+ b_\varepsilon(x)) uv\,d\sigma
\end{gathered}
$$
gives a norm in $H^{1}(\Omega)$, equivalent to the usual one for all $\varepsilon \in [0, \varepsilon_0]$ and some $0<\varepsilon_0\leqslant 1$.
\end{lem}

\begin{proof}
We have $\tn{essinf}_{x\in\Omega}\{\lambda_1+V_\varepsilon\},\tn{essinf}_{x\in\Gamma}\{\lambda_2+b_\varepsilon\}\geqslant m_0$ and  by \eqref{conv_p} it follows that
$$
m_0\|u\|_{H^1(\Omega)}\leqslant \|u\|_{X_\varepsilon^\frac{1}{2}} \leqslant M_0 \|u\|_{H^1(\Omega)},
$$
where $\sup_{0\leqslant\varepsilon\leqslant\varepsilon_0}\{\|p_\varepsilon\|_{L^\infty(\Omega)},\|\lambda_1+V_\varepsilon\|_{L^\infty(\Omega)},\|\lambda_2+b_\varepsilon\|_{L^\infty(\Gamma)}\}\leqslant M_0$.
\end{proof}

Now, if $u$ is a solution of \eqref{eq_reaction_diffusion}, multiplying  \eqref{eq_reaction_diffusion} by a  test function $\varphi$ and integrating by parts over $\Omega$, we obtain
$$
\begin{gathered}
\int_\Omega \partial_tu^\varepsilon\varphi\,dx+\int_\Omega p_\varepsilon(x)\nabla u^\varepsilon\nabla\varphi\,dx+\int_\Omega (\lambda_1+V_\varepsilon)u^\varepsilon \varphi\,dx \\
=\int_\Omega f^\varepsilon (u^\varepsilon)\varphi\,dx+\int_\Gamma \Big[ g^\varepsilon (\gamma_{tr}(u^\varepsilon)) - ( \lambda_2 + b_\varepsilon(x) ) \gamma_{tr}(u^\varepsilon) \Big] \gamma_{tr}(\varphi)\,d\sigma.
\end{gathered}
$$
%But
%$$
%\pin{g^\varepsilon (\gamma_{tr}(u^\varepsilon)),\gamma_{tr}(\varphi)}_{\Gamma}=\int_\Gamma g^\varepsilon (\gamma_{tr}(u^\varepsilon))\gamma_{tr}(\varphi)\,d\sigma=\pin{\gamma_{tr}(g^\varepsilon (u^\varepsilon)),\gamma_{tr}(\varphi)}_{-1,1},
%$$
%then
%$$
%\pin{\partial_tu^\varepsilon,\varphi}_{-1,1}+\pin{A_\varepsilon u^\varepsilon,\varphi}_{-1,1}=\pin{f^\varepsilon (u^\varepsilon),\varphi}_{-1,1}+\pin{g^\varepsilon (\gamma_{tr}(u^\varepsilon)),\gamma_{tr}(\varphi)}_{-1,1}.
%$$
Hence, to obtain appropriated semilinear formulation for \eqref{eq_reaction_diffusion}, we consider nonlinear maps $h^\varepsilon: H^1(\Omega) \mapsto H^{-\beta}(\Omega)$, $\beta>1/2$, given by 
\begin{equation}\label{defi_h}
\pin{h^\varepsilon (u),\varphi}=\int_\Omega f^\varepsilon (u)\varphi\,dx+\int_\Gamma \Big[ g^\varepsilon (\gamma_{tr}(u)) - ( \lambda_2 + b_\varepsilon(x) ) \gamma_{tr}(u) \Big] \gamma_{tr}(\varphi)\,d\sigma.
\end{equation}  

It is well-known that if  $f^\varepsilon$ and $g^\varepsilon$ are $\mathcal{C}^2$ bounded functions with derivatives up to second order bounded and if $\frac{1}{2} < \beta<1$, then $h^\varepsilon$ is a well defined Nemitsk\u{\i}i function which is Fr\'echet continuously differentiable  (see for instance \cite{Pereira2007}).
Thus, from now on, we can take $A_\varepsilon: X^{1-\beta/2}_\varepsilon = H^{2-\beta}(\Omega) \mapsto X^{-\beta/2}_\varepsilon = H^{-\beta}(\Omega)$, for some fixed $\beta \in (1/2,1)$, rewriting \eqref{eq_reaction_diffusion} in the following abstract form
\begin{equation}\label{semilinear_problem}
\begin{cases}
\dfrac{du^\varepsilon}{dt}+A_\varepsilon u^\varepsilon=h^\varepsilon(u^\varepsilon),\\
u^\varepsilon(0)=u^\varepsilon_0 \in H^1(\Omega),\quad \varepsilon\in [0,\varepsilon_0].
\end{cases}
\end{equation}
Notice that our base space here is $H^{-\beta}(\Omega)$ and we have just formulated the problem in the phase space $H^1(\Omega)$. 

We assume that $f^\varepsilon,g^\varepsilon$ satisfy the standard growth, sign, and dissipative conditions uniformly in $\varepsilon$ according to \cite{J.M.Arrieta2000}. 
First, we suppose there exist $B_0$, $C_0 \in \mathbb{R}$ and $B_1$, $C_1>0$ such that for all $u \in \mathbb{R}$ and $\varepsilon \geq 0$ the following holds  
\begin{equation}\label{diss_cond}
\begin{gathered}
u f^\varepsilon(u) \leq - C_0 u^2 + C_1 |u| \quad \textrm{ and } \quad 
u g^\varepsilon(u) \leq - B_0 u^2 + B_1 |u|.
\end{gathered}
\end{equation}
Next, we assume the first eigenvalue of the following eigenvalue problem is positive
\begin{equation}\label{grwoth_cond}
\begin{gathered}
-\mbox{div}(p_0(x)\nabla u)+(\lambda_1+V_0(x) + C_0) u = \mu u \quad \tn{in}\ \Omega \\
\dfrac{\partial u}{\partial \vec{n}_{p_0}}+(\lambda_2+b_0(x) + B_0)u = 0 \quad \tn{on}\ \Gamma.
\end{gathered}
\end{equation}

Consequently, for each $\varepsilon\in [0,\varepsilon_0]$, we get from \cite{J.M.Arrieta2000} that \eqref{semilinear_problem} is globally well-posed and its solutions are classical and continuously differentiable concerning the initial data. Therefore we are able to consider in $H^1(\Omega)$ the family of nonlinear semigroups $\{T_\varepsilon(t);\ t\geqslant 0\}_{\varepsilon\in [0,\varepsilon_0]}$ defined by $T_\varepsilon(t)=u^\varepsilon(t,u^\varepsilon_0)$, $t\geqslant 0$, where $u^\varepsilon(t,u^\varepsilon_0)$ is the solution of \eqref{semilinear_problem} through $u^\varepsilon_0\in H^1(\Omega)$ satisfying the variation of constants formula,
\begin{equation}\label{semigroupononlinear}
T_\varepsilon(t)u_0^\varepsilon   = e^{-A_{\varepsilon}t}u_0^\varepsilon+\int_0^{t} e^{-A_\varepsilon (t-s)}h^\varepsilon(T_\varepsilon(s))\,ds,\quad t\geqslant 0.
\end{equation}
Notice that $e^{-A_{\varepsilon}(\cdot)}$ is the linear semigroup generated by the operator $-A_\varepsilon$. The existence of attractor and uniform bounds for semigroups $\{T_\varepsilon(t);\ t\geqslant 0\}$ associated with \eqref{semilinear_problem} are also established in \cite{J.M.Arrieta2000}  where it was proved the following result.

\begin{theo}\label{uniform_bounds}
The nonlinear semigroup $\{T_\varepsilon(t);\ t\geqslant 0\}$ associated with \eqref{semilinear_problem} has global attractor $\mathcal{A}_\varepsilon$ in $H^1(\Omega)$. Furthermore
\begin{equation*}
\sup_{\varepsilon\in[0,\varepsilon_0]}\sup_{w\in\mathcal{A}_\varepsilon}\|w\|_{H^1(\Omega)}<\infty\quad \hbox{ and }\quad \sup_{\varepsilon\in[0,\varepsilon_0]}\sup_{w\in\mathcal{A}_\varepsilon}\|w\|_{L^\infty(\Omega)}<\infty,
\end{equation*}
for some $\varepsilon_0>0$.
\end{theo}

\begin{rem}\label{f_globally_Lipschitz}
Once the uniform bound in $L^\infty(\Omega)$ for the attractors has been obtained, we may perform a cutoff to the nonlinearities $f^\varepsilon$ and $g^\varepsilon$ so that the new nonlinearity is globally Lipschitz and globally bounded with bounded derivatives up to second order with Lipschitz derivative. See that it coincides with the original one in a $L^\infty$-neighborhood of all the attractors and is strictly dissipative outside this neighborhood. This guarantees that the system with the new nonlinearities has an attractor that coincides exactly with the original ones. 
\end{rem}

In order to apply the theory of rate of convergence of attractors of the \ref{SecA}, we need to obtain estimates of the nonlinear semigroup $T_\varepsilon(\cdot)$ in a bunded set $D$ given by 
$$
D=\cup_{\varepsilon\in[0,\varepsilon_0]} \mathcal{A}_\varepsilon.
$$ 
Notice that, by Theorem \ref{uniform_bounds} if $u\in D$, then there is a constant $C>0$ independent of $\varepsilon$ such that $\|u\|_{H^1(\Omega)},$ $\|u\|_{H^\frac{1}{2}(\Gamma)}\leqslant C$.

\begin{theo}
There exists $C>0$ independent of $\varepsilon$ such that, for all $u$, $v\in D$, 
\begin{multline}\label{estimate_h}
\max\{ \|h^\varepsilon(u)-h^0(v)\|_{H^{-\beta}(\Omega)},\|(h^\varepsilon)^\prime(u)-(h^0)^\prime(v)\|_{\mathcal{L}(H^1(\Omega),H^{-\beta}(\Omega))} \}\\ \leqslant C \Big\{ \|u-v\|_{H^1(\Omega)}+ \tau(\varepsilon) + \kappa(\varepsilon) + \xi(\varepsilon) \Big\}.
\end{multline}
\end{theo}
\begin{proof}
It follows from \eqref{conv_f} that
\begin{align*}
\|f^\varepsilon(u)-f^0(v)\|_{L^2(\Omega)}&\leqslant \|f^\varepsilon(u)-f^0(u)\|_{L^2(\Omega)}+\|f^0(u)-f^0(v)\|_{L^2(\Omega)}\\
&\leqslant  \kappa(\varepsilon) +L_{f^0}\|u-v\|_{L^2(\Omega)},
\end{align*}
where $L_{f^0}$ is the Lipschitz constant of $f^0$ according with the Remark \ref{f_globally_Lipschitz}.

Similarly, we can prove that 
$$
\|(f^\varepsilon)^\prime(u)-(f^0)^\prime(v)\|_{L^2(\Omega)}\leqslant \kappa(\varepsilon) +L_{(f^0)^\prime}\|u-v\|_{L^2(\Omega)},
$$
The same argument shows that 
$$
\|g^\varepsilon(u)-g^0(v)\|_{L^2(\Gamma)}\leqslant \xi(\varepsilon) +L_{g^0}\|u-v\|_{L^2(\Gamma)}
$$
and
$$
\|(g^\varepsilon)^\prime(u)-(g^0)^\prime(v)\|_{L^2(\Gamma)}\leqslant \xi(\varepsilon) +L_{(g^0)^\prime}\|u-v\|_{L^2(\Gamma)}.
$$

By assumptions, we have $\frac{1}{q}+\frac{1}{r_1}+\frac{1}{2}=1$ where $H^1(\Gamma)\subset L^{r_1}(\Gamma)$. By \eqref{conv_v}, Holder inequality and generalized Holder inequality, we have
\begin{align*}
|\pin{h^\varepsilon(u)-h^0(v),\varphi}_{-\beta,\beta}|&\leqslant \int_\Omega |f^\varepsilon(u)-f^0(v)||\varphi|\,dx+\int_\Gamma |g^\varepsilon(u)-g^0(v)||\varphi|\,d\sigma\\
&+\int_\Gamma |\lambda_2||u-v||\varphi|\,d\sigma+\int_\Gamma |b_\varepsilon-b_0||u||\varphi|\,d\sigma+\int_\Gamma |b_0||u-v||\varphi|\,d\sigma\\
&\leqslant L_{f^0}\|u-v\|_{L^2(\Omega)}\|\varphi\|_{H^1(\Omega)}+\kappa(\varepsilon)\|\varphi\|_{H^1(\Omega)}\\
&+ L_{g^0}\|u-v\|_{L^2(\Gamma)}\|\varphi\|_{H^1(\Gamma)}+\xi(\varepsilon)\|\varphi\|_{H^1(\Gamma)}\\ 
&+|\lambda_2|\|u-v\|_{L^2(\Gamma)}\|\varphi\|_{H^1(\Gamma)}\\
&+\|b_\varepsilon-b_0\|_{L^q(\Gamma)}\|u\|_{L^2(\Gamma)}\|\varphi\|_{L^{r_1}(\Gamma)}\\
&+\|b_0\|_{L^q(\Gamma)}\|u-v\|_{L^2(\Gamma)}\|\varphi\|_{L^{r_1}(\Gamma)}.
\end{align*}
In the same way, we obtain the estimate for the derivative. 
\end{proof}

\section{Convergence of the resolvent operators}\label{Sec3}

In this section, we study the convergence of solutions of the elliptic problem associated with \eqref{eq_reaction_diffusion} and their consequences. We start stating some result of \cite{Arrieta} concerning the elliptic problem  $A_\varepsilon u^\varepsilon = f$, where $A_\varepsilon$ is the operator defined in \eqref{ell_operator}-\eqref{ell_operator2} and $f\in L^2(\Omega)$. We are interested in the convergence of the resolvent operators $A_\varepsilon^{-1}$ to $A_0^{-1}$, as we will see this convergence is of order $\|p_\varepsilon-p_0\|_{L^\infty(\Omega)}+\eta(\varepsilon)$. 

The spectrum $\sigma(-A_\varepsilon)$ of $-A_\varepsilon$, ordered and counting multiplicity is given by 
\begin{equation}\label{spectro_A}
...-\lambda^\varepsilon_m<-\lambda^\varepsilon_{m-1}<...<-\lambda_0^\varepsilon
\end{equation}
with $\{\varphi_i^\varepsilon\}_{i=0}^\infty$ the matching eigenfunctions, for any $\varepsilon\in[0,\varepsilon_0]$. We consider the spectral projection onto the space generated by the first $m$ eigenvalues, i.e., if $\omega$ is an appropriated closed curve in the resolvent set $\rho(-A_0)$ of $-A_0$ around $\{-\lambda_0^0,...,-\lambda_{m-1}^0 \}$, then we define the spectral projection
\[
Q_\varepsilon=\frac{1}{2\pi i}\int_{\omega} (\mu+A_\varepsilon)^{-1}\,d\mu,\quad \varepsilon\in[0,\varepsilon_0].
\]

\begin{theo}\label{resolvent_convergence} It is valid the following statements.
\begin{itemize}
\item[(i)] There exists  constant $C>0$ independent of $\varepsilon$ such that
\begin{equation}
\sup_{\varepsilon\in[0,\varepsilon_0]}\|A_\varepsilon^{-1}\|_{\mathcal{L}(H^{-\beta}(\Omega),H^1(\Omega))}\leqslant C.
\end{equation}

\item[(ii)] There exists  constant $C>0$ independent of $\varepsilon$ such that
\begin{equation}
\|A_\varepsilon^{-1}-A_0^{-1}\|_{\mathcal{L}(H^{-\beta}(\Omega),H^1(\Omega))}\leqslant C(\|p_\varepsilon-p_0\|_{L^\infty(\Omega)}+\eta(\varepsilon)).
\end{equation}
\item[(iii)] There exists $\phi\in (\frac{\pi}{2},\pi)$ such that for all 
\[
\mu \in\Sigma_{\nu,\phi}=\{\mu\in\mathbb{C}:|\tn{arg}(\mu+\nu)|\leqslant \phi \}\setminus\{\mu\in\mathbb{C}:|\mu+\nu|\leqslant  r\},
\]
for some $r,\nu>0$, it is valid
\begin{equation}\label{resolvent_estimate2}
\|(\mu+A_\varepsilon)^{-1}-(\mu+A_0)^{-1}\|_{\mathcal{L}(H^{-\beta}(\Omega),H^1(\Omega))}\leqslant C(\|p_\varepsilon-p_0\|_{L^\infty(\Omega)}+\eta(\varepsilon)),
\end{equation}
where $C=C(\mu)>0$ is a constant independent of $\varepsilon$.
\item[(iv)] The family of projections $Q_\varepsilon$ converges to $Q_0$ in the uniform operator topology as $\varepsilon\to 0^+$ and 
\begin{equation}\label{theo_spectral_projection}
\|Q_\varepsilon-Q_0\|_{\mathcal{L}(H^{-\beta}(\Omega),H^1(\Omega))}\leqslant C(\|p_\varepsilon-p_0\|_{L^\infty(\Omega)}+\eta(\varepsilon)),
\end{equation}
where $C>0$ independent of $\varepsilon$.
\item[(v)] If $\lambda_0\in \sigma(A_0)$ and if there is $\lambda_\varepsilon\in\sigma(A_\varepsilon)$ such that  $\lambda_\varepsilon\overset{\varepsilon\to 0^+}\longrightarrow \lambda_0$, then 
$$
|\lambda_\varepsilon-\lambda_0|\leqslant C(\|p_\varepsilon-p_0\|_{L^\infty(\Omega)}+\eta(\varepsilon)).
$$ 
\end{itemize}
\end{theo}
\begin{proof}
Although here we are considering the Neumann boundary condition \eqref{ell_operator2} and the potentials $V_\varepsilon$ in \eqref{ell_operator}, the proof follows the same arguments used in \cite[Section 3]{Arrieta} where the authors have considered the problem with Dirichlet homogeneous boundary conditions without potentials.
\end{proof}

Now, we prove a generalization of Lemma 3.1 of \cite{Arrieta} where we add the boundary terms in the bilinear form $a_\varepsilon$.  We denote 
$$
b_\varepsilon(u,\varphi)=\int_\Gamma (\lambda_2+b_\varepsilon) u\varphi\,d\sigma,
$$
for all $u,\varphi\in L^2(\Gamma)$.

In what follows, to avoid heavy notation, we omit the trace operator $\gamma_{tr}$.

\begin{lem}\label{elliptic_lemma}
For $v \in H^{-\beta}(\Omega)$, with $\beta \in (1/2, 1]$, $\|v\|_{H^{-\beta}(\Omega)} \leqslant 1$ and $\varepsilon\in [0,\varepsilon_0]$, let $u^\varepsilon$ be the solution of the problem
$\pin{A_\varepsilon u^\varepsilon,\cdot}_{-\beta,\beta}+b_\varepsilon(u^\varepsilon,\cdot) =\pin{v,\cdot}_{-\beta,\beta}$. Then, there is a constant $C>0$ independent of $\varepsilon$ such that
\begin{equation}\label{lemma_estimate}
\|u^\varepsilon-u^0\|_{H^1(\Omega)}\leqslant C[\|p_\varepsilon-p_0\|_{L^\infty(\Omega)}+\eta(\varepsilon)+ \tau(\varepsilon)].
\end{equation}
\end{lem}
\begin{proof}
It follows from \eqref{condlam} that $a_\varepsilon+b_\varepsilon$ is symmetric, continuous and coercive thus, the problem $\pin{A_\varepsilon u^\varepsilon,\cdot}_{-\beta,\beta}+b_\varepsilon(u^\varepsilon,\cdot) =\pin{v,\cdot}_{-\beta,\beta}$ is solvable. Let $u^\varepsilon$ be the solution. Then
\begin{equation}\label{weak_solution}
\int_\Omega p_\varepsilon\nabla u^\varepsilon\nabla\varphi \,dx+\int_\Omega (\lambda_1+V_\varepsilon) u^\varepsilon\varphi\,dx+\int_\Gamma (\lambda_2+b_\varepsilon) u^\varepsilon\varphi\,d\sigma= \pin{v,\varphi}_{-\beta,\beta}
\end{equation}
for all  test function $\varphi \in H^1(\Omega)$ and $\varepsilon\in [0,\varepsilon_0]$.
 
If we take $\varphi=u^\varepsilon-u^0$ in \eqref{weak_solution} for $\varepsilon>0$ and $\varepsilon=0$, we obtain 
\begin{multline*}
\int_\Omega p_\varepsilon\nabla u^\varepsilon(\nabla u^\varepsilon-\nabla u^0)\,dx+\int_\Omega (\lambda_1+V_\varepsilon) u^\varepsilon(u^\varepsilon-u^0)\,dx \\ +\int_\Gamma (\lambda_2+b_\varepsilon) u^\varepsilon(u^\varepsilon-u^0)\,d\sigma = \pin{v,u^\varepsilon-u^0}_{-\beta,\beta}
\end{multline*}
and
\begin{multline*}
\int_\Omega p_0\nabla u^0(\nabla u^\varepsilon-\nabla u^0)\,dx+\int_\Omega (\lambda_1+V_0) u^0(u^\varepsilon-u^0)\,dx \\ +\int_\Gamma (\lambda_2+b_0) u^0(u^\varepsilon-u^0)\,d\sigma = \pin{v,u^\varepsilon-u^0}_{-\beta,\beta}.
\end{multline*}
Now, subtracting the above expressions, we have
\[ \begin{split}
&\int_\Omega p_\varepsilon|\nabla u^\varepsilon-\nabla u^0|^2\,dx  +\int_\Omega (\lambda_1+V_\varepsilon)|u^\varepsilon-u^0|^2\,dx+\int_\Gamma (\lambda_2+b_\varepsilon)|u^\varepsilon-u^0|^2\,d\sigma \\
& = \int_\Omega (p_0-p_\varepsilon) \nabla u^0(\nabla u^\varepsilon-\nabla u^0) \,dx+\int_\Omega (V_0-V_\varepsilon) u^0(u^\varepsilon-u^0)\,dx+\int_\Gamma (b_0-b_\varepsilon) u^0(u^\varepsilon-u^0)\,d\sigma.
\end{split} \]
%which implies, by H\"older's inequality and Lemma \ref{equivalence_norm}, that  
%\[ \begin{split}
%\|u^\varepsilon-&u^0\|_{H^1(\Omega)}^2  \leqslant \int_\Omega (p_0-p_\varepsilon) \nabla u^0(\nabla u^\varepsilon-\nabla u^0) \,dx+\int_\Omega (V_0-V_\varepsilon) u^0(u^\varepsilon-u^0)\,dx \\ & \qquad +\int_\Gamma (b_0-b_\varepsilon) u^0(u^\varepsilon-u^0)\,d\sigma \\
%& \leqslant \| \nabla u^0 \|_{L^2(\Omega)} \|\nabla u^\varepsilon-\nabla u^0\|_{L^2(\Omega)} \|p_\varepsilon-p_0\|_{L^\infty(\Omega)} + \| u^0 \|_{L^2(\Omega)} \|V_\varepsilon-V_0\|_{L^2(\Omega)} \|u^\varepsilon-u^0\|_{L^2(\Omega)}\\ & \qquad + \| u^0 \|_{L^2(\Gamma)} \|b_\varepsilon-b_0\|_{L^\infty(\Gamma)} \|u^\varepsilon-u^0\|_{L^2(\Gamma)} \\
%& \leqslant C \|u^\varepsilon-u^0\|_{H^1(\Omega)} \left( \|p_\varepsilon-p_0\|_{L^\infty(\Omega)} + \|V_\varepsilon-V_0\|_{L^\infty(\Omega)} + \|b_\varepsilon-b_0\|_{L^\infty(\Gamma)} \right),
%\end{split} \]

By assumptions, we have $\frac{1}{p}+\frac{1}{r_0}+\frac{1}{2}=1$ and $\frac{1}{q}+\frac{1}{r_1}+\frac{1}{2}=1$ where $H^1(\Omega)\subset L^{r_0}(\Omega)$ and $H^1(\Gamma)\subset L^{r_1}(\Gamma)$. By Holder inequality and generalized Holder inequality, we have
\[ \begin{split}
\|u^\varepsilon-&u^0\|_{H^1(\Omega)}^2  \leqslant \int_\Omega (p_0-p_\varepsilon) \nabla u^0(\nabla u^\varepsilon-\nabla u^0) \,dx+\int_\Omega (V_0-V_\varepsilon) u^0(u^\varepsilon-u^0)\,dx \\ & \qquad +\int_\Gamma (b_0-b_\varepsilon) u^0(u^\varepsilon-u^0)\,d\sigma \\
& \leqslant \| \nabla u^0 \|_{L^2(\Omega)} \|\nabla u^\varepsilon-\nabla u^0\|_{L^2(\Omega)} \|p_\varepsilon-p_0\|_{L^\infty(\Omega)} + \| u^0 \|_{L^{r_0}(\Omega)} \|V_\varepsilon-V_0\|_{L^p(\Omega)} \|u^\varepsilon-u^0\|_{L^2(\Omega)}\\ & \qquad + \| u^0 \|_{L^{r_1}(\Gamma)} \|b_\varepsilon-b_0\|_{L^q(\Gamma)} \|u^\varepsilon-u^0\|_{L^2(\Gamma)} \\
& \leqslant C \|u^\varepsilon-u^0\|_{H^1(\Omega)} \left( \|p_\varepsilon-p_0\|_{L^\infty(\Omega)} + \|V_\varepsilon-V_0\|_{L^p(\Omega)} + \|b_\varepsilon-b_0\|_{L^q(\Gamma)}\right),
\end{split} \]
where $C$ depends on $u^0$ (and then, on $v$) and embedding constants, but is independent of $\varepsilon$.  Finally, from \eqref{conv_p} and \eqref{conv_v}, we obtain \eqref{lemma_estimate}.

\end{proof}

The equilibrium solutions of \eqref{semilinear_problem} are those which are independent of time, i.e., for $\varepsilon \in [0,\varepsilon_0]$, they are the solutions of the elliptic problem  $ A_\varepsilon u^\varepsilon-h^\varepsilon(u^\varepsilon)=0$. We denote by $\mathcal{E}_\varepsilon$ the set of the equilibrium solutions of \eqref{semilinear_problem}  and we say that $u^\varepsilon_* \in \mathcal{E}_\varepsilon$ is  hyperbolic if
$$\sigma(A_\varepsilon-(h^\varepsilon)'(u_*^\varepsilon))\cap \{\mu\in\mathbb{C}\,:\, Re(\mu)=0\}=\emptyset.$$
Since hyperbolicity of the equilibrium is a quite common property for reaction-diffusion equations, we assume w.l.g. $\mathcal{E}_0$ is composed of finitely many hyperbolic equilibrium points. 

The next results prove that the family $\{\mathcal{E}_\varepsilon\}_{\varepsilon\in [0,\varepsilon_0]}$ is continuous at $\varepsilon= 0$, thus for $\varepsilon$ sufficiently small, $\mathcal{E}_\varepsilon$ is composed by a finite number of hyperbolic equilibrium points.

We have $u_*^\varepsilon\in\mathcal{E}_\varepsilon$ and $u_*^0\in \mathcal{E}_0$ if and only if, 
\begin{equation}\label{eq_eq1}
u_*^0=(A_0+R_0)^{-1}[h^0(u_*^0)+R_0u_*^0]\quad\tn{and}\quad u_*^\varepsilon=(A_\varepsilon+R_\varepsilon)^{-1}[h^\varepsilon(u_*^\varepsilon)+R_\varepsilon u_*^\varepsilon],
\end{equation}
where $R_0=-(h^0)'(u_*^0)$ and $R_\varepsilon=-(h^\varepsilon)'(u_*^\varepsilon)$. 

We denote $\bar{A}_\varepsilon=A_\varepsilon+R_\varepsilon$, $\bar{A}_0=A_\varepsilon+R_0$, $B_\varepsilon=R_\varepsilon A_\varepsilon^{-1}$ and $B_0=R_0 A_0^{-1}$. We assume initially $\bar{A}_\varepsilon$ invertible. This property will be a consequence of the hyperbolicity of $u^\varepsilon_\ast$.

\begin{prop}\label{prop_conv_BB}
There exists a positive constant $C$ independent of $\varepsilon$ such that,
\begin{itemize}
\item[(i)] $\sup_{\varepsilon\in[0,\varepsilon_0]}\|(A_\varepsilon+R_\varepsilon)^{-1}\|_{\mathcal{L}(H^1(\Omega),H^{-\beta}(\Omega))}\leqslant C.$
\item[(ii)] $\sup_{\varepsilon\in[0,\varepsilon_0]}\|(I+B_\varepsilon)^{-1}\|_{\mathcal{L}(H^1(\Omega),H^{-\beta}(\Omega))}\leqslant C.$
\item[(iii)] $\|B_\varepsilon-B_0\|_{H^1(\Omega)}\leqslant C(\|p_\varepsilon-p_0\|_{L^\infty(\Omega)}+\eta(\varepsilon)+ \tau(\varepsilon)+\kappa(\varepsilon)+\xi(\varepsilon))$.
\end{itemize}
\end{prop}
\begin{proof} The proofs of (i) and (ii) follow the same arguments of \cite[Lemma 4.3]{Arrieta}. To prove (iii) notice that since $u^\varepsilon_\ast$ and $u^0_\ast$ are equilibrium points then $u^\varepsilon_\ast,u^0_\ast\in D$. 

By definition
\begin{align*}
\|B_\varepsilon-B_0\|_{H^1(\Omega)}&=\|R_\varepsilon A_\varepsilon^{-1}-R_0 A_0^{-1} \|_{H^1(\Omega)}  \\
&\leqslant \|R_\varepsilon-R_0\|_{\mathcal{L}(H^1(\Omega),H^{-\beta}(\Omega))}\|A_\varepsilon^{-1}\|_{\mathcal{L}(H^1(\Omega),H^{-\beta}(\Omega))}\\
&+\|R_0\|_{\mathcal{L}(H^1(\Omega),H^{-\beta}(\Omega))}\|A_\varepsilon^{-1}-A_0^{-1}\|_{\mathcal{L}(H^1(\Omega),H^{-\beta}(\Omega))}.
\end{align*}
The result follows from Theorem \eqref{resolvent_convergence} and \eqref{estimate_h}.
\end{proof}

\begin{theo}\label{theo_conv_eq}
It is valid the following convergence
\begin{equation}
\|\bar{A}_\varepsilon^{-1}-\bar{A}_0^{-1}\|_{\mathcal{L}(H^{-\beta}(\Omega),H^1(\Omega))}\leqslant  C(\|p_\varepsilon-p_0\|_{L^\infty(\Omega)}+\eta(\varepsilon)+\tau(\varepsilon)+\kappa(\varepsilon)+\xi(\varepsilon)).
\end{equation}
\end{theo}
\begin{proof}It is valid the following identity (\cite[Propositin 6.6]{Arrieta})
\begin{equation}
\bar{A}_\varepsilon^{-1}-\bar{A}_0^{-1}=(A_\varepsilon^{-1}-A_0^{-1})(I+B_0)^{-1}-A_\varepsilon^{-1}(I+B_0)^{-1}(B_\varepsilon-B_0)(I+B_\varepsilon)^{-1}
\end{equation}
The result follows from Theorem \eqref{resolvent_convergence} and Proposition \eqref{prop_conv_BB}.
\end{proof}

\begin{theo}\label{rate_equilibrium_points}
Let $u_*^0\in\mathcal{E}_0$. Then for each $\varepsilon>0$ sufficiently small, there is $\delta>0$ such that the equation $A_\varepsilon u-h^\varepsilon(u)=0$ has unique solution $u_*^\varepsilon\in \{u\in H^1(\Omega);\|u-u_*^0\|_{H^1(\Omega)}\leqslant \delta\}$. Moreover
\begin{equation}\label{rate_of_equilibrium}
\|u_*^\varepsilon-u_*^0\|_{H^1(\Omega)}\leqslant C[\|p_\varepsilon-p_0\|_{L^\infty(\Omega)}+\eta(\varepsilon)+ \tau(\varepsilon)+\kappa(\varepsilon)+\xi(\varepsilon)].
\end{equation} 
\end{theo}
\begin{proof}
The proof that $u^\varepsilon_\ast$ is an isolated hyperbolic equilibrium point is standard, see  \cite[Chapter 14]{A.N.Carvalho2010}. Here we need to prove the estimate \eqref{rate_of_equilibrium}. 

It follows from  \eqref{eq_eq1}, that $u^\varepsilon_\ast$ and $u^0_\ast$ are such that
\[ \begin{split}
\|u_*^\varepsilon-u_*^0\|_{H^1(\Omega)}&= \|(A_\varepsilon+R_\varepsilon)^{-1}[h^\varepsilon(u_*^\varepsilon)+R_\varepsilon u_*^\varepsilon]  -(A_0+R_0)^{-1}[h^0(u_*^0)+R_0u_*^0 ] \|_{H^1(\Omega)}\\
&\leqslant\| (A_\varepsilon+R_\varepsilon)^{-1}[h^\varepsilon(u_*^\varepsilon)+R_\varepsilon u_*^\varepsilon] -(A_\varepsilon+R_\varepsilon)^{-1}[h^0(u_*^0)+R_0 u_*^0]\|_{H^1(\Omega)} \\
&+ \|(A_\varepsilon+R_\varepsilon)^{-1}[h^0(u_*^0)+R_0 u_*^0]-  (A_0+R_0)^{-1}[h^0(u_*^0)+R_0u_*^0 ]  \|_{H^1(\Omega)}\\
&\leqslant\| (A_\varepsilon+R_\varepsilon)^{-1}[h^\varepsilon(u_*^\varepsilon)-h^0(u_*^0)+R_\varepsilon u_*^\varepsilon-R_0 u_*^0]\|_{H^1(\Omega)} \\
&+ \|(A_\varepsilon+R_\varepsilon)^{-1}-(A_0+R_0)^{-1}[h^0(u_*^0)+R_0u_*^0 ]  \|_{H^1(\Omega)}.
\end{split}
\]
By Theorem \ref{theo_conv_eq}, we have 
\[ \begin{split}
\|u_*^\varepsilon-u_*^0\|_{H^1(\Omega)}&\leqslant\| (A_\varepsilon+R_\varepsilon)^{-1}[h^\varepsilon(u_*^\varepsilon)-h^0(u_*^0)+R_\varepsilon u_*^\varepsilon-R_0 u_*^0]\|_{H^1(\Omega)} \\
&+ C(\|p_\varepsilon-p_0\|_{L^\infty(\Omega)}+ \tau(\varepsilon)+\kappa(\varepsilon)+\xi(\varepsilon)).
\end{split}
\]

Now we denote $z^\varepsilon=h^\varepsilon(u_*^\varepsilon)-h^0(u_*^0)+R_\varepsilon u_*^\varepsilon-R_0u_*^0$. Then we can write
\begin{align*}
z^\varepsilon &=h^\varepsilon(u_*^\varepsilon)-h^0(u_*^0)+R_\varepsilon u_*^\varepsilon-R_0u_*^0\\
&=h^\varepsilon(u_*^\varepsilon)-h^\varepsilon(u_*^0)+R_\varepsilon u_*^\varepsilon-R_\varepsilon u_*^0\\
&+h^\varepsilon(u_*^0)-h^0(u_*^0)+R_\varepsilon u_*^0-R_0u_*^0.
\end{align*}
Since $h^\varepsilon$ is continuously differentiable, by a  Mean Value Theorem type (see \cite[Lemma 4.1]{Arrieta}) for all $\delta>0$, we can take $\varepsilon_0$ sufficiently small such that 
$$
\|h^\varepsilon(u_*^\varepsilon)-h^\varepsilon(u_*^0)+R_\varepsilon u_*^\varepsilon-R_\varepsilon u_*^0\|_{H^1(\Omega)}\leqslant\delta \|u_*^\varepsilon-u_*^0\|_{H^1(\Omega)}
$$
and by \eqref{estimate_h}, we obtain  
$$
\|h^\varepsilon(u_*^0)-h^0(u_*^0)+R_\varepsilon u_*^0-R_0u_*^0\|_{H^1(\Omega)}\leqslant C(\tau(\varepsilon)+\kappa(\varepsilon)+\xi(\varepsilon)).
$$
Thus,
$$
\|(A_\varepsilon+R_\varepsilon)^{-1}z^\varepsilon\|_{H^1(\Omega)}\leqslant \delta \|(A_\varepsilon+R_\varepsilon)^{-1}\|_{\mathcal{L}(H^{-\beta}(\Omega),H^1(\Omega)}\|u_*^\varepsilon-u_*^0\|_{H^1(\Omega)}+C(\tau(\varepsilon)+\kappa(\varepsilon)+\xi(\varepsilon)). 
$$

Putting all estimates together and taking $\delta$ sufficiently small,  we have   
$$
\|u_*^\varepsilon-u_*^0\|_{H^1(\Omega)}\leqslant C[\|p_\varepsilon-p_0\|_{L^\infty(\Omega)}+\eta(\varepsilon)+\tau(\varepsilon)+\kappa(\varepsilon)+\xi(\varepsilon)]+\frac{1}{2}\|u_*^\varepsilon-u_*^0\|_{H^1(\Omega)}.
$$
\end{proof}

As a consequence of the Theorem \ref{rate_equilibrium_points}, we have the following result.  
  
\begin{cor}
The family  $\{\mathcal{E}_\varepsilon\}_{\varepsilon\in (0,\varepsilon_0]}$ is continuous at $\varepsilon=0$. That is, if $\mathcal{E}_0=\{u_*^{0,1},...,u_*^{0,k}\}$ then for $\varepsilon_0$ sufficiently small, $\mathcal{E}_\varepsilon=\{u_*^{\varepsilon,1},...,u_*^{\varepsilon,k}\}$  and 
$$
\|u_*^{\varepsilon,i}-u_*^{0,i}\|_{H^1(\Omega)}\leqslant C[\|p_\varepsilon-p_0\|_{L^\infty(\Omega)}+\eta(\varepsilon)+\tau(\varepsilon)+\kappa(\varepsilon)+\xi(\varepsilon)],\quad i=1,...,k.
$$
\end{cor}

\begin{rem}
The rate of convergence of the resolvent operators  $\|A_\varepsilon^{-1}-A_0^{-1}\|_{\mathcal{L}(H^{-\beta}(\Omega),H^1(\Omega))}$ is a natural quantity to estimate the continuity of eigenvalues and equilibrium points. As we can see in Theorem \ref{rate_equilibrium_points}, this convergence implies the continuity of equilibrium points according to \cite{Carvalho2006}.  
\end{rem}

\section{Convergence of linear and nonlinear semigroups}\label{Sec5}

In this section, we obtain the rate of convergence for families of linear and nonlinear semigroups. It is in this section that we see the arising of the exponent $l$ in the rate \eqref{final_est_1}. The main difficulty is the presence of singularities in time when we use the variation of the constants formula due to the immersions between the fractional power spaces and the base space $L^2$. Our solution to this problem is to make an interpolation using appropriate exponents to improve the exponential decay of the linear semigroup.

\begin{theo}\label{linear_semigroup_estimate12}
There exist constants $\alpha>0$ and $C>0$ independent of $\varepsilon$ such that
\begin{equation}\label{linear_semigroup_estimate}
\|e^{-A_\varepsilon t}-e^{-A_0 t}\|_{\mathcal{L}(H^{-\beta}(\Omega),H^1(\Omega))}\leqslant C e^{-\alpha t}[\|p_\varepsilon-p_0\|_{L^\infty(\Omega)}+\eta(\varepsilon)]^{2\theta}t^{-\frac{1+\beta}{2}-\theta},
\end{equation}
where $t>0$, $\varepsilon\in (0,\varepsilon_0]$ and $\theta \in (0,\frac{1}{2}]$.
\end{theo}
\begin{proof}
For each $\varepsilon\in [0,\varepsilon_0]$, the operators $A_\varepsilon$ are self-adjoint and by Theorem \ref{resolvent_convergence},  $A_\varepsilon^{-1}$ converges to $A_0^{-1}$.  Thus we can define an appropriate closed rectifiable simple curve with trace in the resolvent set of $A_0$ such that, for $\alpha<\lambda_1^0$, the first eigenvalue of $A_0$, and we can choose  $\varepsilon$ sufficiently small (we still denote $\varepsilon\leqslant \varepsilon_0$) and a constant $C=C(\alpha)>0$ such that
\begin{equation}\label{linear_semigroup_estimate1}
\|e^{-A_\varepsilon t}\|_{\mathcal{L}(H^{-\beta}(\Omega),H^1(\Omega))}\leqslant Ce^{-\alpha t}t^{-\frac{1+\beta}{2}},
\end{equation} 
for $t>0$ and $\varepsilon\in [0,\varepsilon_0]$, see \cite{Lpires1} for details on how to get the above estimates. Therefore we obtain
\begin{equation}\label{linear_semigroup_estimate11}
\begin{split}
\|e^{-A_\varepsilon t}-e^{-A_0 t}\|_{\mathcal{L}(H^{-\beta}(\Omega),H^1(\Omega))} &\leqslant  \|e^{-A_\varepsilon t}\|_{\mathcal{L}(H^{-\beta}(\Omega),H^1(\Omega))}+\|e^{A_0 t}\|_{\mathcal{L}(H^{-\beta}(\Omega),H^1(\Omega))} \\  
&\leqslant Ce^{-\alpha t}t^{-\frac{1+\beta}{2}}.
\end{split}
\end{equation}
On the other hand, the linear semigroup is given by
$$
e^{-A_\varepsilon t}=\frac{1}{2\pi i} \int_\omega e^{\mu t}(\mu+A_\varepsilon)^{-1}\, d\mu,
$$
where $\varepsilon\in [0,\varepsilon_0]$ and $\omega$ is the boundary of sector $\Sigma_{\nu,\phi}$ as in Theorem \ref{resolvent_convergence}, oriented in such a way that the imaginary part of $\mu$  increases as $\mu$ runs in $\omega$. Hence
$$
 \|e^{-A_\varepsilon t}-e^{-A_0 t}\|_{\mathcal{L}(H^{-\beta}(\Omega),H^1(\Omega))} \leqslant  \frac{1}{2\pi i} \int_\omega e^{\mu t}|(\mu+A_\varepsilon)^{-1}-(\mu+A_0)^{-1}|\, |d\mu|.
 $$
Taking a appropriate parametrization of $\omega$ as \cite[Lemma 3.9]{Santamaria2014}, we obtain by \eqref{resolvent_estimate2} that
\begin{equation}\label{linear_semigroup_estimate2}
\|e^{-A_\varepsilon t}-e^{-A_0 t}\|_{\mathcal{L}(H^{-\beta}(\Omega),H^1(\Omega))}\leqslant Ce^{-\alpha t}[\|p_\varepsilon-p_0\|_{L^\infty(\Omega)}+\eta(\varepsilon)]t^{-1}.
\end{equation}
Notice that the term $t^{-1}$ causes a singularity at $t=0$. Then we interpolate \eqref{linear_semigroup_estimate11} and \eqref{linear_semigroup_estimate2} with $1-2\theta$ and $2\theta$, respectively, to obtain \eqref{linear_semigroup_estimate}. 
\end{proof}

On the convergence of nonlinear semigroups, we have:

\begin{theo}\label{Theorem_semigroup_nonlinear_estimate}
Let $u^\varepsilon,u^0\in H^1(\Omega)$ and $\theta \in (0,\frac{1}{2}]$, then there exist constants $C>0$ and $L> 0$ such that 
\begin{equation}\label{semigroup_nonlinear_estimate}
\begin{split}
&\|T_\varepsilon(t)u^\varepsilon-T_0(t)u^0\|_{H^1(\Omega)} \\ 
&\leqslant Ce^{Lt}t^{-\frac{1+\beta}{2}-\theta}[\|u^\varepsilon-u^0\|_{H^1(\Omega)}+[\|p_\varepsilon-p_0\|_{L^\infty(\Omega)}+\eta(\varepsilon)]^{2\theta}+\tau(\varepsilon)+\kappa(\varepsilon)+\xi(\varepsilon)],
\end{split}
\end{equation}
for all $t\geqslant 0$ and $\varepsilon\in (0,\varepsilon_0]$.
\end{theo}
\begin{proof}
Since the nonlinear semigroup is given by \eqref{semigroupononlinear}, we have
\[ \begin{split}
\|T_\varepsilon(t)u^\varepsilon-T_0(t)u^0\|_{H^1(\Omega)} 
&\leqslant \|e^{-A_\varepsilon t}(u^\varepsilon-u^0)\|_{H^1(\Omega)} +\|(e^{-A_\varepsilon t}-e^{-A_0 t})u^0\|_{H^1(\Omega)} \\
&+\int_0^{t} \|e^{-A_\varepsilon(t-s)}(h^\varepsilon(T_\varepsilon(s)u^\varepsilon) - h^0(T_0(s)u^0)\|_{H^1(\Omega)}\,ds\\
&+\int_0^{t} \|(e^{-A_\varepsilon(t-s)}-e^{-A_0(t-s)})h^0(T_0(s)u^0)\|_{H^1(\Omega)}\,ds.
\end{split} \]
Consequently from \eqref{estimate_h} and \eqref{linear_semigroup_estimate} we have
\[ \begin{split}
\|T_\varepsilon(t)u^\varepsilon-& T_0(t)u^0\|_{H^1(\Omega)}  \leqslant C[\|u^\varepsilon-u^0\|_{H^1(\Omega)}+(\|p_\varepsilon-p_0\|_{L^\infty(\Omega)}+\eta(\varepsilon))^{2\theta}]t^{-\frac{1}{2}-\theta}e^{-\alpha t}\\
&+C\int_0^{t} (t-s)^{-\frac{1+\beta}{2}}e^{-\alpha(t-s)} \|T_\varepsilon(s)u^\varepsilon - T_0(s)u^0\|_{H^1(\Omega)}\,ds+C[\tau(\varepsilon)+\kappa(\varepsilon)+\xi(\varepsilon)],
\end{split} \]
and using the singular  Gronwall's inequality (see \cite[Chapter 6]{A.N.Carvalho2010}) we obtain \eqref{semigroup_nonlinear_estimate}.
\end{proof}

\begin{rem}
The rate of convergence of eigenvalues, eigenfunctions, and equilibrium points that we have seen in Theorems \ref{resolvent_convergence} and \ref{rate_equilibrium_points}, is better than the rate of convergence of the nonlinear semigroup. In fact, by Theorem \ref{Theorem_semigroup_nonlinear_estimate}, we have
$$
\sup_{u\in \cup_{\varepsilon}\mathcal{A}_\varepsilon} \sup_{t\in[0,1]}   \|T_\varepsilon(t)u-T_0(t)u\|_{H^1(\Omega)} 
\leqslant C([\|p_\varepsilon-p_0\|_{L^\infty(\Omega)}+\eta(\varepsilon)]^{2\theta}+\tau(\varepsilon)+\kappa(\varepsilon)+\xi(\varepsilon)),
$$
for some constant $C$ independent of $\varepsilon$.

Notice that, the infinite dimension gives a time singularity due to spectral behavior of $e^{-A_\varepsilon (\cdot)}$, that is, the appearance of $t^{-1}$ in \eqref{linear_semigroup_estimate2}. But, the singular Gronwall's inequality does not allow the term $t^{-1}$.  This fact affects the above estimate for the convergence of the nonlinear semigroup and consequently will reflect in the rate of convergence of global attractors of problems \eqref{semilinear_problem}. The \ref{SecA} shows how we can transfer the nonlinear semigroup estimate for the global attractor's convergence.

\end{rem}

\section{Convergence of local unstable manifolds}\label{Sec6}

In this section, we use the general theory of invariant manifolds to estimate the convergence of unstable manifolds around a hyperbolic equilibrium point.  We conclude the exponential attraction of attractors $\mathcal{A}_\varepsilon$ concerning semigroup $T_\varepsilon(\cdot)$ in the phase space $H^1(\Omega)$ and we show how the nonlinear boundary conditions affect the estimates obtained in \cite{Arrieta}.

Let $u^\varepsilon_{*}\in\mathcal{E}_\varepsilon$ be an equilibrium for \eqref{semilinear_problem} and consider its linearization  
\begin{equation}\label{pusogalerkin}
\dfrac{dw^\varepsilon}{dt}+\overline{A}_\varepsilon w^\varepsilon=h^\varepsilon(w^\varepsilon+u_{*}^\varepsilon)-h^\varepsilon(u_{*}^\varepsilon)-(h^\varepsilon)'(u_{*}^\varepsilon)w^\varepsilon,
\end{equation}
where $w^\varepsilon=u^\varepsilon-u_{*}^\varepsilon$ and $\overline{A}_\varepsilon=A_\varepsilon-(h^\varepsilon)'(u_{*}^\varepsilon)$. Let $\bar\gamma$ be a smooth, closed, simple, rectifiable curve in the resolvent set of $-\overline{A}_\varepsilon$, oriented counterclockwise evolving the first $m$ positive eigenvalues in \eqref{spectro_A}. Thus, we can choose an $\varepsilon>0$ sufficiently small such that
$\{\bar\gamma\}\subset \rho(-\overline{A}_\varepsilon)$ and then, we define $\overline{Q}_\varepsilon^{+}$ by
$$
\overline{Q}_\varepsilon^{+}=\frac{1}{2\pi i} \int_{\bar\gamma} (\mu+\overline{A}_\varepsilon)^{-1} d\mu.
$$
The operator $\overline{A}_\varepsilon$ is self-adjoint and there is a $\alpha>0$ and $M\geqslant 1$ such that, for all $0\leqslant\varepsilon\leqslant\varepsilon_0$,
\begin{equation*}
\begin{split}
&\|e^{-\overline{A}_\varepsilon t}\overline{Q}_\varepsilon^{+}\|_{\mathcal{L}(H^{-\beta}(\Omega),{H^1(\Omega)})}    \leqslant Me^{\alpha t},\ t\leqslant 0\qquad \hbox{and}\\
&\|e^{-\overline{A}_\varepsilon t}(I-\overline{Q}_\varepsilon^{+})\|_{\mathcal{L}(H^{-\beta}(\Omega),{H^1(\Omega)})}\leqslant Mt^{-\frac{1+\beta}{2}}e^{-\alpha t},\ t>0.
\end{split}
\end{equation*}
Moreover, we can prove all convergence results of Section \ref{Sec3} for $\overline{A}_\varepsilon$ in place of $A_\varepsilon$, for example
$$
\|e^{-\overline{A}_\varepsilon t}-e^{-\overline{A}_0 t}\|_{\mathcal{L}(H^{-\beta}(\Omega),H^1(\Omega))}\leqslant C e^{-\alpha t}[\|p_\varepsilon-p_0\|_{L^\infty(\Omega)}+\eta(\varepsilon)]^{2\theta}t^{-\frac{1+\beta}{2}-\theta},\quad t>0.
$$

Using the decomposition ${H^1(\Omega)}=\overline{Q}_\varepsilon^{+}({H^1(\Omega)})\, +\, (I-\overline{Q}_\varepsilon^{+})({H^1(\Omega)})$, the solution $w^\varepsilon$ of \eqref{pusogalerkin} can be decomposed as $w^\varepsilon=v^\varepsilon+z^\varepsilon$, with $v^\varepsilon=\overline{Q}_\varepsilon^{+}w^\varepsilon$ and  $z^\varepsilon=(I-\overline{Q}_\varepsilon^{+})w^\varepsilon$. Defining operators $B_\varepsilon:=\overline{A}_\varepsilon\overline{Q}_\varepsilon^{+}$ and  $\widetilde{A}_\varepsilon:=\overline{A}_\varepsilon(I-\overline{Q}_\varepsilon^{+})$, we rewrite \eqref{pusogalerkin} as
\begin{equation}\label{sist_acoplado}
\begin{cases}
\dfrac{dv^\varepsilon}{dt}+B_\varepsilon v^\varepsilon=H_\varepsilon(v^\varepsilon,z^\varepsilon)\\
\dfrac{dz^\varepsilon}{dt}+\widetilde{A}_\varepsilon z^\varepsilon=G_\varepsilon(v^\varepsilon,z^\varepsilon),
\end{cases}
\end{equation}
where
\[
H_\varepsilon(v^\varepsilon,z^\varepsilon):=\overline{Q}_\varepsilon^{+}[h^\varepsilon(v^\varepsilon+z^\varepsilon+u_{*}^\varepsilon)-h^\varepsilon(u_{*}^\varepsilon)-(h^\varepsilon)'(u_{*}^\varepsilon)(v^\varepsilon+z^\varepsilon)]
\]
and
\[
G_\varepsilon(v^\varepsilon,z^\varepsilon):=(I-\overline{Q}_\varepsilon^{+})[h^\varepsilon(v^\varepsilon+z^\varepsilon+u_{*}^\varepsilon)-h^\varepsilon(u_{*}^\varepsilon)-(h^\varepsilon)'(u_{*}^\varepsilon)(v^\varepsilon+z^\varepsilon)].
\]
The functions $H_\varepsilon$ and $G_\varepsilon$ are continuously differentiable with $H_\varepsilon(0,0)=0=G_\varepsilon(0,0)\in H^{-\beta}(\Omega)$ and $H_\varepsilon'(0,0)=0=G_\varepsilon'(0,0)\in \mathcal{L}(H^1(\Omega),H^{-\beta}(\Omega))$. Hence, given  $\rho_o>0$, there are $0<\overline{\varepsilon}=\overline{\varepsilon}_{\rho_o}\leqslant \varepsilon_0$ and $\delta=\delta_{\rho_o}>0$ such that if $\|v\|_{\overline{Q}_\varepsilon^{+}{H^1(\Omega)}}+\|z\|_{{H^1(\Omega)}}<\delta$ and $\varepsilon\leqslant\varepsilon_0$, then
\begin{equation}\label{funcoes_hg}
\|H_\varepsilon(v,z)\|_{\overline{Q}_\varepsilon^{+}(H^1(\Omega))}\leqslant\rho_o\  \  \hbox{and}\  \  \|G_\varepsilon(v,z)\|_{{H^{-\beta}(\Omega)}}\leqslant\rho_o;
\end{equation}
\begin{equation}\label{funcoes_hg_1}
\|H_\varepsilon(v,z)-H_\varepsilon(\overline{v},\overline{z})\|_{\overline{Q}_\varepsilon^{+}(H^1(\Omega))}\leqslant\rho_o(\|v-\overline{v}\|_{\overline{Q}_\varepsilon^{+}({H^1(\Omega)})}+\|z-\overline{z}\|_{{H^1(\Omega)}});
\end{equation}
\begin{equation}\label{funcoes_hg_2}
\|G_\varepsilon(v,z)-G_\varepsilon(\overline{v},\overline{z})\|_{H^{-\beta}(\Omega)}\leqslant\rho_o(\|v-\overline{v}\|_{\overline{Q}_\varepsilon^{+}({H^1(\Omega)})}+\|z-\overline{z}\|_{{H^1(\Omega)}}).
\end{equation}

\begin{theo}\label{cont-variedades}
Given $D>0$ and $\Delta>0$, $\vartheta\in (0,1)$ and $\bar\rho>0$  such that
\begin{equation}\label{rho-cond}
\begin{split}
&\rho_o M \alpha^{-\frac12} \Gamma\left( \frac12\right)\leqslant D, \quad \rho_o M \, \Gamma\left(\frac{1}{2}\right)\frac{M(1+\Delta)}{\left(2\alpha-\rho M(1+\Delta)\right)^{\frac12}}
\leqslant \Delta,\\
& \qquad
\rho_o M\alpha^{-\frac{1}{2}}\Gamma\left(\frac{1}{2}\right)\left[1+\frac{\rho_o M(1+\Delta)\alpha^{-\frac12}}{\left(2\alpha-\rho_o M(1+\Delta)\right)^{\frac12}}\right]\leqslant \vartheta
\end{split}
\end{equation}
are satisfied for all $\rho_o \in (0,\bar\rho)$. Assume that $H_\varepsilon$ and $G_\varepsilon$ satisfies \eqref{funcoes_hg}-\eqref{funcoes_hg_2}, with $0<\rho_o\leqslant\bar\rho$ for all $(v,z)\in\overline{Q}_\varepsilon^{+}{H^1(\Omega)}\times(I-\overline{Q}_\varepsilon^{+}){H^1(\Omega)}$. Then, there exists $s^{*}_\varepsilon:\overline{Q}_\varepsilon^{+}{H^1(\Omega)}\to(I-\overline{Q}_\varepsilon^{+}){H^1(\Omega)}$ such that the unstable manifold of $u_{*}^\varepsilon$ is given as the graph of the map $s^{*}_\varepsilon$,
$$
W^u(u_{*}^\varepsilon)=\{(v,z)\in {H^1(\Omega)};\ z=s^{*}_\varepsilon(v),\ v\in\overline{Q}_\varepsilon^{+}{H^1(\Omega)}\}.
$$
The map $s^*_\varepsilon$ satisfy
$$
|\!|\!|s^{*}_\varepsilon|\!|\!|:=\sup_{v\in\overline{Q}_\varepsilon^{+}({H^1(\Omega)})}\|s^{*}_\varepsilon(v)\|_{{H^1(\Omega)}}\leqslant D,\ \ \
\|s^{*}_\varepsilon(v)-s^{*}_\varepsilon(\widetilde{v})\|_{{H^1(\Omega)}}\leqslant\Delta\|v-\widetilde{v}\|_{\overline{Q}_\varepsilon^{+}({H^1(\Omega)})},
$$
and there is  $C>0$ independent of $\varepsilon$ and $0<\theta<1$ such that
\begin{equation}\label{convergence_unstable_manifold}
|\!|\!|s_\varepsilon^{*}-s_0^{*}|\!|\!|\leqslant C[(\|p_\varepsilon-p_0\|_{L^\infty(\Omega)}+\eta(\varepsilon))^{\theta}+\tau(\varepsilon)+\kappa(\varepsilon)+\xi(\varepsilon)].
\end{equation}

Furthermore, given $0<\gamma<\alpha$, there is  $0<\rho_1\leqslant \bar\rho$ and $C>0$, independent of $\varepsilon$, such that, for any solution $[t_0,\infty)\ni t \mapsto (v^\varepsilon(t),z^\varepsilon(t))\in H^1(\Omega)$ of \eqref{sist_acoplado},
\begin{equation}\label{Exp-att}
\|z^\varepsilon(t)-s_\varepsilon^{*}(v^\varepsilon(t))\|_{{H^1(\Omega)}}\leqslant Ce^{-\gamma(t-t_0)}\|z^\varepsilon(t_0)-s_\varepsilon^{*}(v^\varepsilon(t_0))\|_{{H^1(\Omega)}},\ {\rm for  \ all}\ t\geqslant t_0.
\end{equation}
\end{theo}
\begin{proof}
The proof is standard and it has been made in several papers, see for example \cite{Arrieta,A.N.Carvalho2010,Henry1980}. In \cite{Arrieta} it has shown how to estimate the convergence \eqref{convergence_unstable_manifold}. Once we have established the spaces and estimates necessary to the proof in the previous sections,  we outline what has been done in these works.

Consider the set
$$
\Sigma_\varepsilon\!=\!\left\{s:\overline{Q}_\varepsilon^{+}({H^1(\Omega)})\!\to\!(I-Q_\varepsilon^{+})({H^1(\Omega)}): |\!|\!|s|\!|\!|\leqslant D, \, \|s(v)\!-\!s(\widetilde{v})\|_{{H^1(\Omega)}}\!\leqslant \!\Delta\|v-\widetilde{v}\|_{\overline{Q}_\varepsilon^{+}{H^1(\Omega)}}\right\}.
$$
It is not difficult to see that $(\Sigma_\varepsilon,|\!|\!|\cdot|\!|\!|)$ is a complete metric space.

Given $s_\varepsilon \in \Sigma_\varepsilon$ and $\Theta\in \overline{Q}_\varepsilon (H^1(\Omega))$, denote by $v^\varepsilon(t)=\psi(t,\tau,\Theta,s_\varepsilon)$ the solution of
\begin{equation*}
\begin{cases}
\dfrac{dv^\varepsilon}{dt}(t)+B_\varepsilon v^\varepsilon(t)=H_\varepsilon(v^\varepsilon(t), s_\varepsilon^{*}(v^\varepsilon(t))), \ t<\bar\tau\\
v^\varepsilon(\bar\tau)=\Theta.
\end{cases}
\end{equation*}

Define the map $\Psi_\varepsilon:\Sigma_\varepsilon\to \Sigma_\varepsilon$ by
\[
\Psi_\varepsilon(s_\varepsilon)\Theta=\int_{-\infty}^{\bar{\tau}} e^{-\widetilde{A}_\varepsilon (\bar{\tau}-s)}    G_\varepsilon(v^\varepsilon(s),s_\varepsilon(v^\varepsilon(s)))ds.
\]

Notice that, from \eqref{funcoes_hg} and \eqref{rho-cond}, we have $\|\Psi(s_\varepsilon)(\Theta)\|_{{H^1(\Omega)}}\leqslant D$.

Now, if $\Theta$, $\widetilde{\Theta}\in \overline{Q}_\varepsilon(H^1(\Omega))$, $s_\varepsilon, \widetilde{s}_\varepsilon\in\Sigma_\varepsilon$, $v^\varepsilon(t)=\psi(t,\bar\tau,\Theta,s_\varepsilon)$ and $\tilde{v}^\varepsilon(t)=\psi(t,\bar\tau,\Theta,\tilde{s}_\varepsilon)$, it is easy to see that
%\begin{eqnarray*}
%\begin{split}
$$
\phi(t)
\leqslant M \|\Theta-\tilde\Theta\|_{H^1(\Omega)}+ M\rho_o (1+\Delta) \int_t^{\bar\tau}
\phi(s)ds+ M \rho_o \alpha^{-1} |\!|\!|s_\varepsilon -\tilde{s}_\varepsilon|\!|\!|
$$
%\end{split}
%\end{eqnarray*}
where $\phi(t)=e^{-\alpha(t-\bar\tau)}\|v^\varepsilon(t)-\tilde{v}^\varepsilon(t)\|_{H^1(\Omega)}$ and using Gronwall's inequality
$$
\|v^\varepsilon(t)-\tilde{v}^\varepsilon(t)\|_{H^1(\Omega)}\leqslant M\left(\|\Theta-\tilde\Theta\|_{H^1(\Omega)}+\rho_o \alpha^{-1} |\!|\!|s_\varepsilon -\tilde{s}_\varepsilon|\!|\!|
\right)e^{(\alpha-M\rho_o(1+\Delta))(t-\bar\tau)}.
$$
From this, we obtain that
\begin{eqnarray*}
\|\Psi(s_\varepsilon)(\Theta)-\Psi(\widetilde{s}_\varepsilon)(\widetilde{\Theta})\|_{{H^1(\Omega)}}&\leqslant&\rho_o M\alpha^{-\frac{1}{2}}\Gamma\left(\frac{1}{2}\right)\left[1+\frac{\rho_o M(1+\Delta)\alpha^{-\frac12}}{\left(2\alpha-\rho_o M(1+\Delta)\right)^{\frac12}}\right]|\!|\!|s_\varepsilon-\widetilde{s}_\varepsilon|\!|\!|+\\
&+&\frac{\rho_o M^2(1+\Delta)}{\left(2\alpha-\rho_o M(1+\Delta)\right)^{\frac12}}
\Gamma\left(\frac{1}{2}\right)\|\Theta-\widetilde{\Theta}\|_{\overline{Q}_\varepsilon^{+}{H^1(\Omega)}}.
\end{eqnarray*}
and
\begin{equation*}
\|\Psi(s_\varepsilon)(\Theta)-\Psi(\widetilde{s}_\varepsilon)(\widetilde{\Theta})\|_{{H^1(\Omega)}}\leqslant\Delta\|\Theta-\widetilde{\Theta}\|_{\overline{Q}_\varepsilon^{+}{H^1(\Omega)}}+\vartheta|\!|\!|s_\varepsilon-\widetilde{s}_\varepsilon|\!|\!|.
\end{equation*}
Hence, $\Psi_\varepsilon$ is a contraction. Therefore, there a fixed point $s_\varepsilon^{*}=\Psi(s_\varepsilon^{*})$ in $\Sigma_\varepsilon$.

Now, we prove that $\{(v^\varepsilon,s_\varepsilon^{*}(v^\varepsilon));\
v\in\overline{Q}_\varepsilon^{+}{H^1(\Omega)}\}$ is invariant for
\eqref{sist_acoplado}. Let $(v_0^\varepsilon,z_0^\varepsilon)\in
W^u(u_{*}^\varepsilon)$, $z_0^\varepsilon=s_\varepsilon^{*}(v_0^\varepsilon)$.
Denote by $v_{*}^\varepsilon(t)$ the solution of the initial value
problems
\begin{equation}\label{ode_flow}
\begin{cases}
\dfrac{dv^\varepsilon}{dt}+B_\varepsilon v^\varepsilon=H_\varepsilon(v^\varepsilon, s_\varepsilon^{*}(v^\varepsilon))\\
v^\varepsilon(0)=v_0^\varepsilon.
\end{cases}
\end{equation}
This defines a curve $(v_\varepsilon^{*}(t),s_\varepsilon^{*}(v_\varepsilon^{*}(t)))\in W^u(u_{*}^\varepsilon)$, $t\in\mathbb{R}$. Also, the only solution of
$$
z_t^\varepsilon+\widetilde{A}_\varepsilon z^\varepsilon=G_\varepsilon(v_{*}^\varepsilon(t),s_\varepsilon^{*}(v_{*}^\varepsilon(t)))
$$
which remains bounded as $t\to-\infty$ must be
$$
z_{*}^\varepsilon(t)=\int_{-\infty}^te^{\widetilde{A}_\varepsilon(t-s)}G_\varepsilon(v_{*}^\varepsilon(s),s_{\varepsilon}^{*}(v_{*}^\varepsilon(s)))ds=s_\varepsilon^{*}(v_{*}^\varepsilon(t)).
$$
This proves the invariance of the graph of $s^*_\varepsilon$. To prove that the graph of $s^*_\varepsilon$ is the unstable manifold assume the exponential attraction of the graph of $s^*_\varepsilon$ uniformly in $\varepsilon$; i.e., if $u^\varepsilon(t)=z^\varepsilon(t)+v^\varepsilon(t)$ is a solution of \eqref{sist_acoplado} with $v^\varepsilon(t)=\overline{Q}_\varepsilon u^\varepsilon(t)$. If, given $\gamma<\alpha$, there exists $\rho_1>0$ such that
\eqref{Exp-att} holds for any $0<\rho_o\leqslant \rho_1$, it is easy to see that, when $z^\varepsilon(t)$ remains bounded as $t\to -\infty$, it follows that (making $t_0\to -\infty$ in \eqref{Exp-att}) that
$z^\varepsilon(t)=s_\varepsilon^{*}(v^\varepsilon(t))$ for all $t\in\mathbb{R}$.

The proof of \eqref{Exp-att} can be carried out as  \cite{Carvalho}, using the singular Gronwall's inequality.

Finally, we obtain the rate of convergence \eqref{convergence_unstable_manifold}. For $\Theta \in \overline{Q}_\varepsilon(H^1_0(\Omega))$,  we have
\[ \begin{split}
\|s^\varepsilon_*(\Theta)-s^0_*(\Theta)\|_{H^1(\Omega)}
& \leqslant \int_{-\infty}^{\bar\tau} \|e^{-\widetilde{A}_\varepsilon(\tau-r)}G_\varepsilon(v^\varepsilon,s_*^\varepsilon(v^\varepsilon))-e^{-\widetilde{A}_\varepsilon(\tau-r)}G_\varepsilon(v^0,s_*^0(v^0))\|_{H^1(\Omega)}\,dr\\
& + \int_{-\infty}^{\bar\tau} \|e^{-\widetilde{A}_\varepsilon(\tau-r)}G_\varepsilon(v^0,s_*^0(v^0))-e^{-\widetilde{A}_\varepsilon-(\tau-r)}G_0(v^0,s_*^0(v^0))\|_{H^1(\Omega)}\,dr\\
& + \int_{-\infty}^{\bar\tau} \|e^{-\widetilde{A}_\varepsilon-(\tau-r)}G_0(v^0,s_*^0(v^0))-e^{-\widetilde{A}_0(\tau-r)}G_0(v^0,s_*^0(v^0))\|_{H^1(\Omega)}\,dr.
\end{split} \]
Since we can estimate $G_\varepsilon-G_0$ by $Q_\varepsilon^{+}-Q_0^{+}$, these integrals are estimated as in the Theorem \ref{Theorem_semigroup_nonlinear_estimate}, and using the singular Gronwall's inequality. Repeating the argument for $v^\varepsilon-v^0 $  we obtain \eqref{convergence_unstable_manifold}.
\end{proof}

\begin{rem}
The projection $\overline{Q}_\varepsilon^+$ has finite hank hence the unstable manifold can be considered as a finite-dimensional object, in fact, $W^u(u^\varepsilon_\ast)\approx\mathbb{R}^m$ and the semigroup restrict to this manifold is conjugate to that generated by a vector field in $\mathbb{R}^m$ that is, it is given by ordinary equation \eqref{ode_flow}. Notice that the rate of convergence of the nonlinear boundary conditions is added to the rate of convergence of unstable manifolds.  
\end{rem}

\section{Convergence of global attractors}\label{Sec7}

In this section, we prove the main result of this work. We used all previous results to put the problem \eqref{eq_reaction_diffusion} in the conditions of Corollary \ref{equi_attraction_continuity}. Here, we consider the dimension $N\geqslant 2$. The scalar case will be addressed in the next section.

Let $X$ be a Banach space. We denote by $\tn{dist}_H(A,B)$ the Hausdorff semi distance between $A,B\subset X$, defined as
$$
\tn{dist}_H(A,B)=\sup_{a\in A}\inf_{b\in B}\|a-b\|_{X},
$$
and we define the symmetric Hausdorff metric by
$$
\tn{d}_H(A,B)=\max\{\tn{dist}_H(A,B),\tn{dist}_H(B,A)\}.
$$

\begin{defi}\label{continuity_attractor_rate}
We say that a family $\{\mathcal{A}_\varepsilon\}_{\varepsilon\in [0\varepsilon_0]}$ of subsets of $X$ is continuous at $\varepsilon=0$ if
$$
\tn{d}_H(\mathcal{A}_\varepsilon,\mathcal{A}_{0})\overset{\varepsilon\to 0^+}\longrightarrow 0.
$$ 
\end{defi}

\begin{theo}\label{main_result_final_rate}
The family of attractors $\{\mathcal{A}_\varepsilon\}_{\varepsilon\in[0,\varepsilon_0]}$ of \eqref{semilinear_problem} for $N\geqslant 2$ is continuous at $\varepsilon=0$. Moreover, this continuity can be estimated by the following rate of convergence
\begin{equation}\label{attractors_rate_final}
\tn{d}_H(\mathcal{A}_\varepsilon,\mathcal{A}_0)\leqslant C[\|p_\varepsilon-p_0\|_{L^\infty(\Omega)}+\eta(\varepsilon)+\tau(\varepsilon)+\kappa(\varepsilon)+\xi(\varepsilon)]^{l},\quad\mbox{with}\quad 0<l<1,
\end{equation}
where $C$ is a constant independent of $\varepsilon$.
\end{theo}

\begin{proof}
It follows from Theorem \ref{Theorem_semigroup_nonlinear_estimate}  that for $u,v\in H^1(\Omega)$,
$$
\tn{dist}_H(T_\varepsilon(t)u,T_0(t)v)\leqslant Ce^{Lt}(\|u-v\|_{H^1(\Omega)}+\delta(\varepsilon)),\quad t\geqslant 1,
$$
where $\delta(\varepsilon)=[\|p_\varepsilon-p_0\|_{L^\infty(\Omega)}+\eta(\varepsilon)]^{2\theta}+ \tau(\varepsilon)+\kappa(\varepsilon)+\xi(\varepsilon) $, with $C,L>0$ and $0<\theta\leqslant\frac{1}{2}$.

By Theorem \ref{uniform_bounds} the set $D=\cup_{\varepsilon\in [0,\varepsilon_0]} \mathcal{A}_\varepsilon$ is uniformly bounded (in $\varepsilon$) and by \eqref{Exp-att} the attractor $\mathcal{A}_\varepsilon$ attracts  exponentially, i.e.
$$
\sup_{\varepsilon\in [0,\varepsilon_0]}\tn{dist}_H(T_\varepsilon(t)D,\mathcal{A}_\varepsilon)\leqslant \Theta(t),\quad t\geqslant 1
$$ 
where $\Theta(t)=Ce^{-\gamma (t-t_0)}$.

Now applying the Corollary \ref{equi_attraction_continuity} we get
$$
\tn{d}_H(\mathcal{A}_\varepsilon,\mathcal{A}_0)\leqslant C[[\|p_\varepsilon-p_0\|_{L^\infty(\Omega)}+\eta(\varepsilon)]^{2\theta}+ \tau(\varepsilon)+\kappa(\varepsilon)+\xi(\varepsilon)]^{\frac{\gamma}{\gamma+L}},
$$
which implies \eqref{attractors_rate_final} taking $\theta=\frac{1}{2}$ and $l=\frac{\gamma}{\gamma+L}$.
\end{proof}

\section{Scalar reaction-diffusion equations}\label{Sec8}

In this section, we study the problem \eqref{semilinear_problem} for the case $N=1$. In this scalar situation, we will have a family of Morse Smale problems with gap conditions in the eigenvalues. This property allows us to obtain a finite-dimensional invariant manifold and then the restriction of the flow in the attractors is given by ordinary differential equations in the Euclidian Space $\mathbb{R}^M$, where $M=m$ was defined in \eqref{spectro_A}. Once this is done, we will use the results of the \ref{SecB}  to obtain a better rate of convergence than the previous section for the convergence of the attractors of \eqref{semilinear_problem}.  

Consider the scalar parabolic equations 
\begin{equation}\label{esc_eq_reaction_diffusion}
\begin{cases}
\partial_tu^\varepsilon-(p_\varepsilon(x) u^\varepsilon_x)_x+(\lambda_1+V_\varepsilon(x))u^\varepsilon=f^\varepsilon(u^\varepsilon),&\ 0<x<1,\quad t>0, \\
\partial_xu^\varepsilon+(\lambda_2+b_\varepsilon(x))u^\varepsilon=g^\varepsilon(u^\varepsilon),&\ x\in \{0,1\},
\end{cases}
\end{equation}
where $\varepsilon\in[0,\varepsilon_0]$, and $\lambda_1$, $\lambda_2$, $p_\varepsilon, V_\varepsilon, b_\varepsilon, f^\varepsilon$ and $g^\varepsilon$ satisfy \eqref{condlam}, \eqref{conv_p}, \eqref{conv_v} and \eqref{conv_f}.

We can rewrite  \eqref{esc_eq_reaction_diffusion} abstractly as in \eqref{semilinear_problem}, where the nonlinearity is given by $h^\varepsilon:H^1(\Omega)\to H^{-\beta}(\Omega)$ is given by \eqref{defi_h}, for  $\beta\in (\frac{1}{2},1)$ and $\Omega=(0,1)$. Thus, with the assumption \eqref{diss_cond},  \eqref{esc_eq_reaction_diffusion} is well posed in $H^1(\Omega)$ and the solutions through $u_0^\varepsilon$ satisfy \eqref{semigroupononlinear}. Moreover, the nonlinear semigroup has  global attractor $\mathcal{A}_\varepsilon$ in $H^1(\Omega)$, such that $\overline{\bigcup_{\varepsilon\in [0,\varepsilon_0]}\mathcal{A}_\varepsilon}$ is compact.  

Moreover, the scalar equation \eqref{esc_eq_reaction_diffusion} is a Morse-Smale problem (see \cite{Henry1985}) in the sense that the unstable and stable local manifolds (of different equilibrium points) have transversal intersection which implies geometrical structural stability of the phase space. Consequently, the nonlinear semigroups are Morse-Smale semigroup for $\varepsilon\in[0,\varepsilon_0]$.

\begin{theo}\label{esc_continuity_nonlinear_semigroup}
For each $w_0\in \overline{\bigcup_{\varepsilon\in [0,\varepsilon_0]}\mathcal{A}_\varepsilon}$, there is constant $C>0$ independent of $\varepsilon$ such that
$$
\|T_\varepsilon(1)w_0-T_0(1)w_0\|_{H^1(\Omega)}\leqslant C\delta(\varepsilon)|\log(\delta(\varepsilon))|, 
$$
where $\delta(\varepsilon)=\|p_\varepsilon-p_0\|_{L^\infty(\Omega)}+\eta(\varepsilon)+\tau(\varepsilon)+\kappa(\varepsilon)+\xi(\varepsilon)$.
\end{theo}

\begin{proof}
As in the Theorem \ref{linear_semigroup_estimate12} we have for,  $\varepsilon\in [0,\varepsilon_0]$,
$$
\|e^{-A_\varepsilon t}-e^{-A_0 t}\|_{\mathcal{L}(H^{-\beta}(\Omega),H^1(\Omega))}\leqslant Ce^{-\alpha t}t^{-\frac{1+\beta}{2}}
$$
and
$$
\|e^{-A_\varepsilon t}-e^{-A_0 t}\|_{\mathcal{L}(H^{-\beta}(\Omega),H^1(\Omega))}\leqslant Ce^{-\alpha t}[\|p_\varepsilon-p_0\|_{L^\infty(\Omega)}+\eta(\varepsilon)]t^{-1}.
$$

Notice that the terms $ t^{-\frac{1+\beta}{2}}$ and $t^{-1}$ in the estimates above originate a singularity in the formula of the variation of the constants. This is the main difficulty in estimating nonlinear semigroups. In the Theorem \ref{Theorem_semigroup_nonlinear_estimate} we performed an interpolation of these terms together with the rate of convergence of resolvent operators which resulted in the considerable loss in the rate of convergence of attractors \eqref{attractors_rate_final}. In the situation where the limiting problems are Morse-Smale, the authors in \cite{Santamaria2017} had the same problem, however, they used the following estimate (placed in our context). If we denote  $l_\varepsilon(t)=\min\{t^{-\frac{1+\beta}{2}},\|p_\varepsilon-p_0\|_{L^\infty(\Omega)}+\eta(\varepsilon)] t^{-1}\}$, then 
\[
\int_{-\infty}^{\bar\tau} l_\varepsilon(\bar\tau-r)e^{-\alpha(\bar\tau-r)}\,dr\leqslant C[\|p_\varepsilon-p_0\|_{L^\infty(\Omega)}+\eta(\varepsilon)]|\log(\|p_\varepsilon-p_0\|_{L^\infty(\Omega)}+\eta(\varepsilon))|.
\]

Since the nonlinear semigroup is given by \eqref{semigroupononlinear}, then for $0<t\leqslant 1$, we have
\begin{multline*}
\|T_\varepsilon(t)w_0-T_0(t)w_0\|_{H^1(\Omega)} \leqslant \|(e^{-A_\varepsilon t}-e^{-A_0 t})w_0\|_{H^1(\Omega)} \\ +\int_0^{t} \|e^{-A_\varepsilon(t-s)}h^\varepsilon(T_\varepsilon(s)w_0)-e^{-A_0(t-s)}h^0(T_0(s)w_0)\|_{H^1(\Omega)}\,ds,
\end{multline*}
Now, as in Theorem \ref{Theorem_semigroup_nonlinear_estimate}, we obtain 
$$
\|T_\varepsilon(t)w_0-T_0(t)w_0\|_{H^1(\Omega)}\leqslant C\delta(\varepsilon)|\log(\delta(\varepsilon))|+C\delta(\varepsilon)|\log(\delta(\varepsilon))|e^{Kt},
$$
where $K>0$. Now the result follows taking $t=1$.
\end{proof}

We saw in the Section \ref{Sec3} that, for each $\varepsilon\in[0,\varepsilon_0]$, the spectrum $\sigma(-A_\varepsilon)$ of $-A_\varepsilon$, ordered and counting multiplicity is given by $...-\lambda^\varepsilon_m<-\lambda^\varepsilon_{m-1}<...<-\lambda_0^\varepsilon$. Moreover  is true the following gap condition
$$
|\lambda^\varepsilon_m-\lambda^\varepsilon_{m-1}| \to \infty \quad\tn{as}\quad m\to\infty. 
$$
This property enables us to find a finite dimension invariant manifold as well as in \cite{Santamaria2014} and \cite{Santamaria2013}.

\begin{theo}\ For sufficiently large $m$ and  $\varepsilon$  small there is an invariant manifold $\mathcal{M}_\varepsilon$ for the problem \eqref{semilinear_problem} given by
$$
\mathcal{M}_\varepsilon=\{u^\varepsilon\in H^1(\Omega)\,;\, u^\varepsilon = Q_\varepsilon u^\varepsilon+s_{*}^\varepsilon(Q_\varepsilon u^\varepsilon)\},\quad \varepsilon\in[0,\varepsilon_0],
$$ 
where, $Q_\varepsilon$ is the spectral projection and  $s_\ast^\varepsilon:Y_\varepsilon\to Z_\varepsilon$ is a Lipschitz continuous map satisfying
\[
|\!|\!|s_\ast^\varepsilon-s_\ast^0 |\!|\!|=\sup_{v\in Y_\varepsilon}\|s_\ast^\varepsilon(v)-s_\ast^0(v)\|_{H^1(\Omega)}\leqslant C\delta(\varepsilon)|\log(\delta(\varepsilon))|,
\]
where $\delta(\varepsilon)=\|p_\varepsilon-p_0\|_{L^\infty(\Omega)}+\eta(\varepsilon)+\tau(\varepsilon)+\kappa(\varepsilon)+\xi(\varepsilon)$, $C>0$  is  constant independent of $\varepsilon$ and $Y_\varepsilon=Q_\varepsilon H^1(\Omega)$ and $Z_\varepsilon=(I-Q_\varepsilon)H^1(\Omega)$.
The invariant manifold $\mathcal{M}_\varepsilon$ is exponentially attracting and the  attractor $\mathcal{A}_\varepsilon$ of the problem \eqref{semilinear_problem} lies  in $\mathcal{M}_\varepsilon$. The flow on $\mathcal{A}_\varepsilon$ is given by
$$
u^\varepsilon(t)=v^\varepsilon(t)+s_\ast^\varepsilon(v^\varepsilon(t)), \quad t\in\mathbb{R},
$$ 
where $v^\varepsilon(t)$ satisfy 
\begin{equation}\label{ode_reduced_A}
\dfrac{dv^\varepsilon}{dt}+A_\varepsilon^+v^\varepsilon=Q_\varepsilon h^\varepsilon(v^\varepsilon+ s_\ast^\varepsilon(v^\varepsilon(t))),
\end{equation}
where $A_\varepsilon^+=A_\varepsilon|_{Y_\varepsilon}$.
\end{theo}
\begin{proof}
The proof follows the same steps of the Theorem \ref{cont-variedades}, see \cite{Santamaria2014} and \cite{Santamaria2013} for more details.
\end{proof}

Thus the flow in $\mathcal{A}_\varepsilon$  is given by ordinary equation \eqref{ode_reduced_A}. Since $Q_\varepsilon$ has finite rank that we denote by $M$ and, we can consider $v^\varepsilon\in \mathbb{R}^M$ and $H_\varepsilon(v^\varepsilon)=Q_\varepsilon h^\varepsilon(v^\varepsilon+ s_\ast^\varepsilon(v^\varepsilon(t)))$ a continuously differentiable map in $\mathbb{R}^M$. For each $\varepsilon\in [0,\varepsilon_0]$, we denote $\tilde{T}_\varepsilon=\tilde{T}_\varepsilon(1)$, where $\tilde{T}_\varepsilon(\cdot)$ is the semigroup generated by solution $ v^\varepsilon(\cdot)$  of \eqref{ode_reduced_A} in $\mathbb{R}^M$. We have the following convergences
\begin{equation}\label{estimate_one_map}
\|\tilde{T}_\varepsilon-\tilde{T}_0\|_{C^1(\mathbb{R}^M,\mathbb{R}^M)}\overset{\varepsilon\to 0^+}\longrightarrow 0\quad\tn{and}\quad\|\tilde{T}_\varepsilon-\tilde{T}_0\|_{L^\infty(\mathbb{R}^M,\mathbb{R}^M)}\leqslant C\delta(\varepsilon)|\log(\delta(\varepsilon))|,
\end{equation}
where the last estimate can be proved as well as in Theorem \ref{esc_continuity_nonlinear_semigroup}. 

Therefore, we have a Morse-Smale semigroup in $\mathbb{R}^M$ and using techniques of shadowing see \ref{SecB},  we are ready to prove the main result of this section.

\begin{theo} Let $\mathcal{A}_\varepsilon$, $\varepsilon\in[0,\varepsilon_0]$, be the attractor for the problem \eqref{semilinear_problem}. Then there is constant $C>0$ independent of $\varepsilon$ such that
$$
\tn{d}_H(\mathcal{A}_\varepsilon,\mathcal{A}_0)\leqslant C\delta(\varepsilon)|\log(\delta(\varepsilon))|,
$$
where $\delta(\varepsilon)=\|p_\varepsilon-p_0\|_{L^\infty(\Omega)}+\eta(\varepsilon)+\tau(\varepsilon)+\kappa(\varepsilon)+\xi(\varepsilon)$.
\end{theo}
\begin{proof} For each $\varepsilon\in [0,\varepsilon_0]$ we denote $T_\varepsilon=T_\varepsilon(1)$. Given $u^\varepsilon\in \mathcal{A}_\varepsilon$, by invariance there is $w^\varepsilon\in \mathcal{A}_\varepsilon$ such that $u^\varepsilon=T_\varepsilon w^\varepsilon$ so we can write 
$
w^\varepsilon=Q_\varepsilon w^\varepsilon + s_\ast^\varepsilon(Q_\varepsilon w^\varepsilon),
$
where $Q_\varepsilon w^\varepsilon \in \bar{\mathcal{A}_\varepsilon}$ with $\bar{\mathcal{A}_\varepsilon}=Q_\varepsilon\mathcal{A}_\varepsilon$ the projected attractor in $\mathbb{R}^M$. Thus
\[ \begin{split}
\|u^\varepsilon-u^0\|_{H^1(\Omega)} &=\|T_\varepsilon w^\varepsilon-T_0 w^0\|_{H^1(\Omega)}\leqslant \|T_\varepsilon w^\varepsilon-T_\varepsilon w^0\|_{H^1(\Omega)}+\|T_\varepsilon w^0-T_0 w^0\|_{H^1(\Omega)}\\
& \leqslant C\delta(\varepsilon)|\log(\delta(\varepsilon))|.
\end{split} \]
But 
\[
\begin{split}
&\|w^\varepsilon-w^0\|=\|Q_\varepsilon w^\varepsilon-Q_0 w^0\|_{\mathbb{R}^M}+\|s_\ast^\varepsilon(Q_\varepsilon w^\varepsilon)-s_\ast^0(Q_0 w^0)\|_{H^1(\Omega)}\\
&\leqslant \|Q_\varepsilon w^\varepsilon-Q_0 w^0\|_{\mathbb{R}^M}+\|s_\ast^\varepsilon(Q_\varepsilon w^\varepsilon)-s_\ast^\varepsilon(Q_0 w^0)\|_{H^1(\Omega)}+\|s_\ast^\varepsilon(Q_0 w^0)-s_\ast^0(Q_0 w^0)\|_{H^1(\Omega)}\\
&\leqslant C \|Q_\varepsilon w^\varepsilon-Q_0 w^0\|_{\mathbb{R}^M}+C\delta(\varepsilon)|\log(\delta(\varepsilon))|,
\end{split}
\]
which implies 
$$
\tn{d}_H(\mathcal{A}_\varepsilon,\mathcal{A}_0)\leqslant \tn{d}_H(\bar{\mathcal{A}_\varepsilon},\bar{\mathcal{A}_0})+C\delta(\varepsilon)|\log(\delta(\varepsilon))|.
$$
The result follows by \eqref{estimate_one_map} and Theorem \ref{Shadowing_reaction_diffusion}. 
\end{proof}

\appendix

\section{General theory of rate of convergence of attractors}\label{SecA}
Here we summarize some results of \cite{Carvalho2010}.

Let $\{T_\varepsilon(\cdot)\}_{\varepsilon\in \Lambda}$ be a family of semigroups  on a Banach space $X$, where $\Lambda$ is a topology space. We will assume that the family  $\{T_\varepsilon(\cdot)\}_{\varepsilon\in \Lambda}$ converges in some appropriate sense to $T_{\varepsilon_0}(\cdot)$ as $\varepsilon\to \varepsilon_0$ in $\Lambda$. 

\begin{defi}
If $T_\varepsilon(\cdot)$ is a family of semigroups with attractors $\mathcal{A}_\varepsilon$, then we say that $\{\mathcal{A}_\varepsilon\}_{\varepsilon\in \Lambda}$ is equi-attracting if 
$$
\sup_{\varepsilon\in\Lambda}\tn{dist}_H(T_\varepsilon(t)B,\mathcal{A}_\varepsilon)\overset{t\to \infty}\longrightarrow 0,
$$
for each bounded subset $B$ of $X$.
\end{defi}

We now see that equi-attraction implies the continuity of attractors, we can also obtain a rate of convergence of attractors concerning the underlying parameters.

\begin{theo}\label{equi_attraction}
Let $T_\varepsilon(\cdot)$ be a family of semigoups with attractors $\mathcal{A}_\varepsilon$ and set $D=\cup_{\varepsilon\in\Lambda} \mathcal{A}_\varepsilon$. If there is a strictly decreasing function $\Theta:[t_0,\infty)\to (0,\infty)$ such that 
$$
\sup_{\varepsilon\in\Lambda}\tn{dist}_H(T_\varepsilon(t)D,\mathcal{A}_\varepsilon)\leqslant \Theta(t),\quad t\geqslant t_0
$$ 
and, there are $C,L>0$ and a function $\delta(\varepsilon)\to 0$ as $\varepsilon\to \varepsilon_0$ such that
$$
\tn{dist}_H(T_\varepsilon(t)x,T_0(t)y)\leqslant Ce^{Lt}(\|x-y\|_X+\delta(\varepsilon)),\quad t\geqslant t_0,\,\,\,x,y\in D.
$$
Then
$$
\tn{d}_H(\mathcal{A}_\varepsilon,\mathcal{A}_0)\leqslant \min_{\nu\in \Theta([0,\infty))} 2\{Ce^{L\Theta^{-1}(\nu)}\delta(\varepsilon)+\nu)\}.
$$
\end{theo}

\begin{proof}
For $t\geqslant t_0$ we have 
\[
\begin{split}
\tn{dist}_H(\mathcal{A}_\varepsilon,\mathcal{A}_0) & \leqslant \tn{dist}_H(T_\varepsilon(t)\mathcal{A}_\varepsilon,T_0(t)\mathcal{A}_\varepsilon)+\tn{dist}_H(T_0(t)\mathcal{A}_\varepsilon,\mathcal{A}_0)\\
&\leqslant \sup_{x\in\mathcal{A}_\varepsilon}\tn{dist}_H(T_\varepsilon(t)x,T_0(t)x)+\tn{dist}_H(T_0(t)\mathcal{A}_\varepsilon,\mathcal{A}_0)\\
&\leqslant Ce^{L t}\delta(\varepsilon)+\Theta(t).
\end{split}
\]
Given $\nu \in \Theta([0,\infty))$ and $t=\Theta^{-1}(\nu)$, we get
\[
\tn{dist}_H(\mathcal{A}_\varepsilon,\mathcal{A}_0)\leqslant Ce^{L \Theta^{-1}(\nu) }\delta(\varepsilon)+\nu. 
\] 
\end{proof}

In particular, if the equi-attraction is of exponential order, then we have an optimal rate of convergence. 

\begin{cor}\label{equi_attraction_continuity}
Assume that the conditions of Theorem \ref{equi_attraction} are satisfied with $\Theta(t)=ce^{-\gamma t}$ for some $c\geqslant 1$, $\gamma>0$ and for all $t\in [t_0,\infty)$. Then there is  constant $\bar{c}>0$ independent of $\varepsilon$, such that 
$$
\tn{d}_H(\mathcal{A}_\varepsilon,\mathcal{A}_0)\leqslant \bar{c}\delta(\varepsilon)^{\frac{\gamma}{\gamma+L}}. 
$$ 
\end{cor}
\begin{proof}
Take $\Theta^{-1}(\nu)=\log(\frac{c}{\nu})^{\frac{1}{\gamma}}$, we obtain
\begin{equation}\label{eq.est.equi.attraction}
\tn{d}_H(\mathcal{A}_\varepsilon,\mathcal{A}_0)\leqslant 2\min_{\nu\in (0,c]}\Big[\delta(\varepsilon)\Big(\frac{c}{\nu}\Big)^\frac{L}{\gamma}+\nu\Big].
\end{equation}
The minimum of the right-hand side of \eqref{eq.est.equi.attraction} occurs when $\nu=c(\frac{L}{\gamma})^\frac{\gamma}{\gamma+L}\delta(\varepsilon)^\frac{\gamma}{\gamma+L}$. Since the left-hand size of \eqref{eq.est.equi.attraction} is independent of $\nu$, it follows that
$$
\tn{d}_H(\mathcal{A}_\varepsilon,\mathcal{A}_0)\leqslant 2c\Big[ \Big(\frac{L}{\gamma}\Big)^\frac{-L}{\gamma+L}+\Big(\frac{L}{\gamma}\Big)^\frac{\gamma}{\gamma+L}\Big]\delta(\varepsilon)^\frac{\gamma}{\gamma+L}.
$$
\end{proof}

\begin{example}
Consider the following problem 
\[
\begin{cases}
\varepsilon\dfrac{d^2x}{dt^2}+\dfrac{dx}{dt}=\mu x+f(x),\quad x\in\mathbb{R}^N\\ x(0)=x_0\in\mathbb{R}^N\\ \dfrac{dx}{dt}(0)=v_0\in\mathbb{R}^N.
\end{cases}
\]
Assume that $\varepsilon\in [0,1]$, $\mu\geqslant 1$, $f:\mathbb{R}^N\to\mathbb{R}^N$ is a $C^1$-function which is globally Lipschitz, globally bounded and with symmetric Jacobian matrix at every point. If we write the above equation in the form of a system with variables $x$ and $v=\varepsilon\frac{dx}{dt}$ we have 
\begin{equation}\label{second_ode}
\begin{cases}
\frac{d}{dt}\Big[\begin{smallmatrix} x \\v \end{smallmatrix}\Big]=\Big[\begin{smallmatrix} 0 & \frac{1}{\varepsilon}I \\ -\mu I & -\frac{1}{\varepsilon}I \end{smallmatrix}\Big]\Big[\begin{smallmatrix} x \\v \end{smallmatrix}\Big]+\Big[\begin{smallmatrix} 0 \\f(x) \end{smallmatrix}\Big],\quad \Big[\begin{smallmatrix} x \\v \end{smallmatrix}\Big]\in \mathbb{R}^N\times\mathbb{R}^N, \\ \Big[\begin{smallmatrix} x \\v \end{smallmatrix}\Big](0)=\Big[\begin{smallmatrix} x_0 \\v_0 \end{smallmatrix}\Big]\in \mathbb{R}^N\times\mathbb{R}^N.
\end{cases}
\end{equation}
The solutions of \eqref{second_ode} are globally defined and the solution operator family $T_\varepsilon(\cdot)$ defines a semigroup in $Z=\mathbb{R}^N\times\mathbb{R}^N$ which has  attractor $\mathcal{A}_\varepsilon$. The problem \eqref{second_ode} can be rewritten as  
$$
\begin{cases}
\frac{d}{dt} \Big[\begin{smallmatrix} I & I \\ -\varepsilon\mu I & 0 \end{smallmatrix}\Big]\Big[\begin{smallmatrix} x \\v \end{smallmatrix}\Big]=-\Big[\begin{smallmatrix} x \\v \end{smallmatrix}\Big]+\Big[\begin{smallmatrix} I & I \\ -\varepsilon\mu I & 0\end{smallmatrix}\Big] \Big[\begin{smallmatrix} 0 \\f(x) \end{smallmatrix}\Big],\quad \Big[\begin{smallmatrix} x \\v \end{smallmatrix}\Big]\in \mathbb{R}^N\times\mathbb{R}^N, \\ \Big[\begin{smallmatrix} x \\v \end{smallmatrix}\Big](0)=\Big[\begin{smallmatrix} x_0 \\ \varepsilon v_0 \end{smallmatrix}\Big]\in \mathbb{R}^N\times\mathbb{R}^N.
\end{cases}
$$

Since $\varepsilon$ goes to zero, one would expect that the dynamical properties of \eqref{second_ode} are given by
$$
\begin{cases}
\frac{d}{dt} \Big[\begin{smallmatrix} I & I \\ 0 & 0 \end{smallmatrix}\Big]\Big[\begin{smallmatrix} x \\v \end{smallmatrix}\Big]=-\Big[\begin{smallmatrix} x \\ v \end{smallmatrix}\Big]+\Big[\begin{smallmatrix} I & I \\ 0 & 0 \end{smallmatrix}\Big] \Big[\begin{smallmatrix} 0 \\f(x) \end{smallmatrix}\Big],\quad \Big[\begin{smallmatrix} x \\ v \end{smallmatrix}\Big]\in \mathbb{R}^N\times\mathbb{R}^N, \\ \Big[\begin{smallmatrix} x \\v \end{smallmatrix}\Big](0)=\Big[\begin{smallmatrix} x_0 \\ \varepsilon v_0 \end{smallmatrix}\Big]\in \mathbb{R}^N\times\mathbb{R}^N
\end{cases}
$$
which corresponds to $v=0$ and
\begin{equation}\label{second_ode_ode}
\begin{cases}
\dfrac{dx}{dt}=-\mu x+f(x)\\ x(0)=x_0\in \mathbb{R}^N.
\end{cases}
\end{equation} 
Notice that the solutions for \eqref{second_ode_ode} are globally defined and the solution operator family $R _0(\cdot)$ defines a semigroup in $\mathbb{R}^N$. To compare the dynamics of these two problems we should find a way to see the dynamics of \eqref{second_ode_ode} in $Z$. That is done simply by defining
$$
T_0(t)\Big[\begin{smallmatrix} x_0 \\ v_0 \end{smallmatrix}\Big]=\Big[\begin{smallmatrix} R_0(t)x_0\\  0 \end{smallmatrix}\Big],\quad t>0, \quad T_0(0)=I,
$$
and noting that $T_0(\cdot)$ is a semigroup (singular at zero) with attractor $\mathcal{A}_0$.
We can prove (see \cite{Carvalho2010}) that there are constants $\varepsilon_0>0$, $\gamma>0$, $t_0>0$ and $C>0$ independent of $\varepsilon$ such that  
$$
\tn{dist}_H\Big(T_\varepsilon(t)\bigcup_{\varepsilon\in [0,\varepsilon_0]}\mathcal{A}_\varepsilon,\mathcal{A}_\varepsilon\Big)\leqslant Ce^{-\gamma t},\quad t\geqslant t_0
$$ 
and for any $\bar{x},\bar{y}\in Z$,
$$
\|T_\varepsilon(t)\bar{x}-T_0(t)\bar{y}\|_Z\leqslant Ce^{L t},\|\bar{x}-\bar{y}\|_Z+\varepsilon^\alpha,\quad t\geqslant t_0,\quad \alpha<1.
$$ 

Then by Corollary \ref{equi_attraction_continuity}
$$
\tn{d}_H(\mathcal{A}_\varepsilon,\mathcal{A}_0)\leqslant C \varepsilon^{\frac{\alpha\gamma}{\gamma+L}}.
$$
\end{example}

\section{Shadowing theory and rate of convergence of attractors}\label{SecB}

In this section we make a brief overview of some results of Shadowing Theory presented in \cite{Santamaria2017} and \cite{Santamaria2013}. This theory enables us to estimate the convergence of attractors by the continuity of nonlinear semigroups. Since the semigroups are given by the variation of constants formula we can use the finite dimension to obtain estimates for the linear semigroup and then estimate the continuity of nonlinear semigroups.

We observe that the arguments work in finite dimension,  therefore applying the techniques used here in more general problems where the dynamics can not be described by an ODE it seems not capable of working successfully.

Let $T:\mathbb{R}^N\to \mathbb{R}^N$ be a continuous function. Recall that the discrete dynamical system generated by $T$ is defined by $T^0=I_{\mathbb{R}^N}$ and, for $k\in\mathbb{N}$, $T^k=T\circ\dots\circ T$ is the $k$th iterate of $T$. The notions of Morse-Smale systems, hyperbolic fixed points, and stable and unstable manifolds for a function $T$ are similar to the continuous case (see \cite{Hale1988}).  

\begin{defi}
A trajectory (or global solution) of the discrete dynamical system generated by $T$ is a sequence $\{x_n\}_{n\in\mathbb{Z}}\subset \mathbb{R}^N$, such that, $x_{n+1}=T(x_n)$, for all $n\in\mathbb{Z}$. 
\end{defi}

\begin{defi}
We say that a sequence $\{x_n\}_{n\in\mathbb{Z}}$ is a $\delta$-pseudo-trajectory of $T$ if 
$$
\|Tx_x-x_{k+1}\|_{\mathbb{R}^N}\leqslant\delta,\ \tn{for all}\ n\in\mathbb{Z}. 
$$
\end{defi}

\begin{defi}
We say that a point $x\in \mathbb{R}^N$ $\varepsilon$-shadows a $\delta$-pseudo-trajectory $\{x_k\}$ on $U\subset \mathbb{R}^N$ if the following inequality holds
\[
\|T^kx-x_k\|_{\mathbb{R}^N}\leqslant \varepsilon,\ \tn{for all}\  k\in\mathbb{Z},
\]
\end{defi}

\begin{defi}
The map $T$ has the Lipschitz Shadowing Property (LpSP) on $U\subset \mathbb{R}^N$, if there are constants $L, \delta_0>0$ such that for any $0< \delta\leqslant \delta_0 $, any $\delta$-pseudo-trajectory of $T$ in $U$ is $(L\delta)$-shadowed by a trajectory of $T$ in $\mathbb{R}^N$, i.e., for any sequence $\{x_k\}_k\subset U\subset \mathbb{R}^N$ with
$$
\|Tx_k-x_{k+1}\|_{\mathbb{R}^N}\leqslant \delta\leqslant \delta_0,\quad k\in\mathbb{Z},
$$
there is a point $x\in X$ such that the following inequality holds
$$
\|T^kx-x_k\|_{\mathbb{R}^N}\leqslant L\delta,\quad k\in\mathbb{Z}.
$$
\end{defi}

Let $T:\mathbb{R}^M\to \mathbb{R}^M$ be a Morse Smale function which has attractor $\mathcal{A}$. Since $\mathcal{A}$ is compact and has all dynamics of the system, we can restrict our attention to a neighborhood $\mathcal{N}(\mathcal{A})$ of $\mathcal{A}$, thus we consider the space $C^1(\mathcal{N}(\mathcal{A}),\mathbb{R}^M)$ with the $C^1$-topology.

The next result can be found in \cite{Santamaria2014}. It describes an application of Shadowing Theory to the rate of convergence of attractors.

\begin{prop}\label{prop_estimate_LPSP}
Let $T_1, T_2:X\to X$ be maps which has  global attractors $\mathcal{A}_1,\mathcal{A}_2$. Assume that $\mathcal{A}_1,\mathcal{A}_2\subset \mathcal{U}\subset X$, that $T_1,T_2$ have both the LpSP on $\mathcal{U},$ with parameters $L,\delta_0$ and $\|T_1-T_2\|_{\mathcal{L}^\infty(\mathcal{U},X)}\leqslant \delta$. Then we have 
$$
\tn{dist}_H(\mathcal{A}_1,\mathcal{A}_2)\leqslant \|T_1-T_2\|_{\mathcal{L}^\infty(\mathcal{U},X)}.
$$
\end{prop}
\begin{proof}
Since $T_1$ and $T_2$ has the LpSP on $\mathcal{U}$. Take a trajectory $\{y_n\}_n$ of $T_2$ in $ \mathcal{A}_2$, then $\{y_n\}_n$ is a $\delta$-pesudo-trajectory of $T_1$ with $\delta=\|T_1-T_2\|_{\mathcal{L}^\infty(\mathcal{U},X)}\leqslant \delta_0$. By LpSP there is a trajectory $\{x_n\}_n\subset X$ of $T_1$ such that $\|x_n-y_n\|_X\leqslant L\delta$ for al $n\in\mathbb{Z}$, hence $\{x_n\}_n\subset \mathcal{A}_1$. Since $\{y_n\}_n$ is arbitrary the result follows. 
\end{proof}

As a consequence of Proposition \ref{prop_estimate_LPSP}, we have the following result.

\begin{theo}\label{Shadowing_reaction_diffusion}
Let $T:\mathbb{R}^M\to \mathbb{R}^M$ be a Morse-Smale function with attractor $\mathcal{A}$. Then there exist constant $L>0$, a neighborhood $\mathcal{N}(\mathcal{A})$ of $\mathcal{A}$ and a neighborhood $\mathcal{N}(T)$ of $T$ in the $C^1(\mathcal{N}(\mathcal{A}),\mathbb{R}^M)$ topology such that, for any $T_1,T_2\in \mathcal{N}(T)$ with attractors  $\mathcal{A}_1$, $\mathcal{A}_2$, respectively, we have
$$
\tn{dist}_H(\mathcal{A}_1,\mathcal{A}_2)\leqslant L\|T_1-T_2\|_{L^\infty(\mathcal{N}(\mathcal{A}),\mathbb{R}^M)}.
$$
\end{theo}

\begin{proof}
In \cite{Pilyugun1999} was proved that a structurally stable dynamical system on a compact manifold has the LpSP and it is known that a Morse-Smale system is structurally stable. Moreover, In \cite{Santamaria2014} was proved that a discrete Morse-Smale semigroup $T$ has the LpSP in a neighborhood $\mathcal{N}(\mathcal{A})$ of its attractor $\mathcal{A}$. Hence, the result follows by the Proposition \ref{prop_estimate_LPSP}.
\end{proof}

\begin{example}\label{Ex_thin_domain}
Consider the set $Q=\{(x,y)\in \mathbb{R}\times\mathbb{R}^{N-1}: 0\leqslant x \leqslant 1,\,\,|y|<1\}$.  The thin domain is defined by $Q_\varepsilon=\{(x,\varepsilon y)\in \mathbb{R}\times\mathbb{R}^{N-1}: (x,y)\in Q\}$, $\varepsilon\in (0,1]$. This domain is obtained by shrinking the set $Q$ by a factor $\varepsilon$ in the $N-1$ direction given by the variable $y$. Note that $Q_\varepsilon$ collapses to a straight segment $(0,1)$ as $\varepsilon$ goes to zero, see Figure \ref{thin_domain_draw}. 
\begin{figure}[!htp]
\begin{center}
\begin{tikzpicture}[line cap=round,line join=round]
\tkzDefPoint(3,0.8){A}
\tkzDefPoint(3,-0.8){B}
\tkzDefPoint(3,0){0}
\tkzDefPoint(5.7,0.5){P}
\tkzDefPoint(6,0){x}
\tkzDefPoint(9,0){1}
\tkzDefPoint(9,0.8){E}
\tkzDefPoint(9,-0.8){F}
%\draw[dashed,thin,name path=elip] (0) ellipse (0.4cm and 0.8cm);
%\draw[dashed,thin,name path=elip] (x) ellipse (0.4cm and 0.8cm);
\draw[name path=elip] (1) ellipse (0.4cm and 0.8cm);
\draw[draw=none,name path=0C] (0)--(1);
\draw (3,-0.8) arc (270:90:0.4cm and 0.8cm);
\draw[dotted] (3,0.8) arc (90:-90:0.4cm and 0.8cm);
\draw (6,-0.8) arc (270:90:0.4cm and 0.8cm);
\draw[dotted] (6,0.8) arc (90:-90:0.4cm and 0.8cm);
\tkzDrawSegments[->](x,P)
\tkzDrawPoints[fill=black](0,x,1)
\tkzDrawSegments(0,1)
\tkzDrawSegments(A,E)
\tkzDrawSegments(B,F)
\draw (6.1,0.4) node{$\varepsilon y$};
\draw (3,-0.3) node{$0$};
\draw (6,-0.2) node{$x$};
\draw (9,-0.3) node{$1$};
\draw (2.4,0.9) node{$Q_\varepsilon$};
\end{tikzpicture}
\caption{Thin Domain}\label{thin_domain_draw}
\end{center}
\end{figure}

We consider the boundary value problem associated with a reaction-diffusion equation
\begin{equation}\label{thin_domain}
\begin{cases}
\partial_tu-\Delta u+\alpha u = f(u)&\tn{ in } Q_\varepsilon\times[0,\infty)\\ 
\dfrac{\partial u}{\partial \nu_\varepsilon}=0&\tn{ on }\partial Q_\varepsilon\times(0,\infty),
\end{cases}
\end{equation}
where $\alpha>0$, $\nu_\varepsilon$ the unit outward normal to $\partial Q_\varepsilon$ and $f:\mathbb{R}\to\mathbb{R}$ a $C^2$-function satisfying
\begin{itemize}
\item[(i)] $|f'(s)|\leqslant C(1+|s|^{\rho-1})$, $s\in\mathbb{R}$, for some $\rho >1$,
\item[(ii)] $f(s)s\leqslant 0$, $|s|\geqslant M$, for some $M>0$.
\end{itemize}
The limit problem of \eqref{thin_domain} as $\varepsilon \to 0$ is given by
\begin{equation}\label{thin_domain_limit}
\begin{cases}
\partial_tu-\partial_x^2u+\alpha u=f(u)\,\,\tn{in }(0,1),\\
\partial_xu(0)=\partial_xu(1)=0.
\end{cases}
\end{equation} 
The equation \eqref{thin_domain} is well posed in $H^1(Q_\varepsilon)$ and the equation \eqref{thin_domain_limit} is well posed in $H^1(0,1)$, i.e., if we define the operators $A_\varepsilon=-\Delta u+\alpha u$ and $A_0u=-\partial_x^2u+\alpha u$, the global solutions of the Cauchy problem of type \eqref{semilinear_problem} with $f^\varepsilon_\Omega=f$ and $g^\varepsilon_\Gamma\equiv0$ generates a nonlinear semigroup $T_\varepsilon(\cdot)$,  and this semigroups has global attractor $\mathcal{A}_\varepsilon$, for all $\varepsilon\in [0,1]$. Moreover, $T_0(\cdot)$ is a Morse-Smale semigroup. 

In order to understand the attractor $\mathcal{A}_0$ in $H^1(Q_\varepsilon)$, we define $E_\varepsilon: H^1(0,1)\to H^1(Q_\varepsilon)$ by $(E_\varepsilon u)(x,y)=u(x)$, for all $u\in H^1(0,1)$. Assuming that the equilibrium points of \eqref{thin_domain_limit} are hyperbolic, was proved in \cite{Santamaria2017}  the following results
\begin{itemize}
\item[(i)] $\|A_\varepsilon^{-1}-E_\varepsilon A_0^{-1}\|_{\mathcal{L}(L^2(0,1),H^1(0,1))}\leqslant C\varepsilon$;
\item[(ii)] $\|T_\varepsilon(1)-E_\varepsilon T_0(1)\|_{\mathcal{L}(H^1(0,1))}\leqslant C \varepsilon|\log(\varepsilon)|$;
\item[(iii)] There is a finite-dimensional invariant manifold $\mathcal{M}_\varepsilon$ given by graph of Lipschitz functions $s_\varepsilon$ such that $\mathcal{A}_\varepsilon\subset \mathcal{M}_\varepsilon$ and the flow can be reduced to finite dimension, i.e., we can consider an ODE generating a semigroup $\overline{T}_\varepsilon(\cdot)$ in $\mathbb{R}^M$ with an attractor $\overline{\mathcal{A}} _\varepsilon$, where $M=\tn{dim}(\mathcal{M}_\varepsilon)$. Moreover, we have the following estimates    
\begin{itemize}
\item[(a)]$\|\overline{T}_\varepsilon(1)-\overline{T}_0(1)\|_{\mathcal{L}(\mathbb{R}^M)}\leqslant C\varepsilon|\log(\varepsilon)|;$
\item[(b)]$\sup_{v\in\mathbb{R}^M} \|s_\varepsilon(v)-s_0(v)\|_{H^1(0,1)}\leqslant C\varepsilon|\log(\varepsilon)|.$
\end{itemize}
\end{itemize}
Therefore the convergence of attractors of \ref{thin_domain} and \ref{thin_domain_limit} can be estimated by 
$$\tn{d}_H(\mathcal{A}_\varepsilon, E_\varepsilon \mathcal{A}_0)\leqslant C \varepsilon^\frac{N+1}{2}|\log(\varepsilon)|,$$
where $C>0$ is a constant independent of $\varepsilon.$
\end{example}

\section*{Data availability statement}

 The data that support the findings of this study are available within the article.

\section*{Conflict of interest}

The authors declare that they have no conflict of interest.

\section*{Funding}

\noindent F. D. M. Bezerra thank the CNPq for financial support through the project \# 303039/2021-3, Brazil;

\noindent  M. C. Pereira thank the CNPq for financial support through the project \# 308950/2020-8 and FAPESP for financial support through the project  \# 2020/14075-6, Brazil.

\bibliographystyle{abbrv}
\bibliography{References}
\end{document}